\documentclass[11pt]{amsart}
\usepackage{amsfonts,amssymb,amsthm,amsmath,amsxtra,amscd,verbatim,eucal}

\usepackage[all]{xy}
\usepackage[dvips]{graphics}

\title{Singularities with respect to Mather-Jacobian discrepancies}
\author{Lawrence Ein}
\author{Shihoko Ishii} 
\address{Department of Mathematics, University of Illinois at Chicago, Chicago, IL 60607-7045, USA}
\email{ein@math.uic.edu}
\address{Graduate school of Mathematical Science, University of Tokyo, Meguro, Tokyo, Japan}
\email{shihoko@ms.u-tokyo.ac.jp}

\thanks{2010\,\emph{Mathematics Subject Classification}.
 Primary 14F18; Secondary 14B05.
\newline
The first author: partially supported by NSF grant DMS-1001336
The second author:
 partially supported by Grant-in-Aid (B) 22340004 }

\def\cO{\mathcal{O}}

\newcommand{\bC}{{\Bbb C}}
\newcommand{\bP}{{\Bbb P}}
\newcommand{\bZ}{{\Bbb Z}}
\newcommand{\bQ}{{\Bbb Q}}
\newcommand{\bG}{{\Bbb G}}
\newcommand{\bR}{{\Bbb R}}
\newcommand{\bN}{{\Bbb N}}
\newcommand{\bA}{{\Bbb A}}

\newcommand{\bx}{{\bf x}}

\renewcommand{\int}{{\operatorname{in}}}
\newcommand{\emb}{{\operatorname{emb}}}
\newcommand{\codim}{\operatorname{codim}}
\newcommand{\spec}{\operatorname{Spec}}

\newcommand{\Hom}{\operatorname{Hom}}
\newcommand{\sing}{\operatorname{Sing}}
\newcommand{\ord}{\operatorname{ord}}
\newcommand{\val}{\operatorname{val}}
\newcommand{\mult}{\operatorname{mult}}
\newcommand{\olx}{{\overline{X}}}
\newcommand{\ola}{{\overline{A}}}
\newcommand{\oa}{{\mathcal O}_A}
\newcommand{\ooa}{{{\mathcal O}_\ola}}
\newcommand{\oaa}{{{\mathcal O}_{A'}}}
\newcommand{\lm}{{{\mathcal{L}}^m}}
\renewcommand{\ln}{{{\mathcal{L}}^n}}
\newcommand{\li}{{{\mathcal{L}}^\infty}}

\let \cedilla =\c

\renewcommand{\o}[0]{{\mathcal O}} 
\newcommand{\ox}{{{\mathcal O}_X}}
\newcommand{\oy}{{{\mathcal O}_Y}}
\newcommand{\oyy}{{{\mathcal O}_{Y'}}}

\newcommand{\hx}{{\widehat X}}
\newcommand{\hk}{{\widehat K}}
\renewcommand{\a}{{\frak{a}}}
\renewcommand{\b}{{\frak{b}}}

\newcommand{\mld}{\operatorname{mld}}
\renewcommand{\j}{{\mathcal J}}
\newcommand{\jx}{{{\mathcal J}_X}}

\newcommand{\hmld}{{\operatorname{mld}_{\operatorname{MJ}}}}
\newcommand{\hj}{{\mathcal J}_{\operatorname{MJ}}}
\newcommand{\hKY}{\widehat K_{Y/X}}

\newcommand{\ha}{a_{\operatorname{MJ}}}
\newcommand{\im}{{\operatorname{Im}}}
\newcommand{\cont}{\operatorname{Cont}}

\def\to {\longrightarrow}

\newtheorem{thm}{Theorem}[section]

\newtheorem{lem}[thm]{Lemma}
\newtheorem{cor}[thm]{Corollary}
\newtheorem{prop}[thm]{Proposition}

\theoremstyle{definition}
\newtheorem{defn}[thm]{Definition}

\newtheorem{say}[thm]{}
\newtheorem{exmp}[thm]{Example}

\newtheorem{rem}[thm]{Remark}

\theoremstyle{remark}

\begin{document}
\maketitle
\begin{abstract}
As is well known, 
the ``usual discrepancy" is defined for a  normal  
$\bQ$-Gorenstein variety. By using this discrepancy we can define a canonical singularity and a log canonical singularity.
In the same way, by using a new notion, Mather-Jacobian discrepancy introduced in recent papers we can define a ``canonical singularity" and a ``log canonical singularity" for 
not necessarily normal or $\bQ$-Gorenstein varieties.
In this paper, we show basic properties of these singularities, behavior of these singularities under  deformations and determine all these singularities of dimension up to 2. 

\end{abstract}

\section{Introduction}
In birational geometry,  canonical, log canonical, terminal and log terminal singularities
play important roles.
These singularities are all normal $\bQ$-Gorenstein singularities and
each step of 
the minimal model program is performed inside  the category of normal $\bQ$-Gorenstein
singularities.
But in turn, from a purely singularity theoretic view point, the normal $\bQ$-Gorenstein
property seems, in some sense, to be an unnecessary restriction  for a singularity to be considered as a 
good singularity, because there are many ``good" singularities without normal 
$\bQ$-Gorenstein property (for example, the cone over the Segre embedding
$\bP^1\times\bP^2\hookrightarrow \bP^5$).

In this paper, we take off the restriction normal $\bQ$-Gorenstein, give definitions of 
``good" singularities which have some compatibilities with the usual canonical, log canonical, terminal and log terminal singularities
and  study  our ``good"
singularities.
To contrast, remember the definition of the usual canonical, log canonical, terminal and log terminal singularities. 
We say that a pair $(X, \a^t)$ consisting of a normal $\bQ$-Gorenstein variety $X$, an ideal $\a
\subset \ox$ and
$t\in \bR_{\geq 0}$  has canonical (resp. log canonical, terminal, log terminal ) 
singularities if, for a log resolution $\varphi:Y\to X$ of $(X, \a)$, the log discrepancy 
$a(E; X, \a^t)$ satisfies the inequality 
$$a(E; X, \a^t):=\ord_E(K_{Y/X})-t \val_E(\a)+1\geq 1\  (\mbox{resp.} \ \geq 0,\ 
> 1, \ > 0)$$
 for every exceptional prime divisor $E$.
We say that $(X, \a^t)$ has klt singularities if the above inequality holds for every
prime divisor on $Y$.
Here we note that the discrepancy divisor $K_{Y/X}=K_Y-\frac{1}{r}\varphi^*(rK_X)$ is well defined
if there is an integer $r$ such that $rK_X$ is a Cartier divisor, which means that 
$X$ is a $\bQ$-Gorenstein variety.

Now, consider a pair $(X, \a^t)$ under a more general setting.
Let $X$ be a  connected reduced equidimensional affine scheme of finite type over
an algebraically closed field $k$ of characteristic zero.
For a log resolution  $\varphi:Y\to X$ of $(X, \a)$ which factors through the Nash 
blow-up, we can define the Mather discrepancy divisor $\hk_{Y/X}$ (Definition \ref{mather}).
For the Jacobian ideal $\jx\subset \ox$ we define the Jacobian discrepancy divisor 
$J_{Y/X}$ by $\oy(-J_{Y/X})=\jx\oy$. 
The combination $\hk_{Y/X}-J_{Y/X}$ is called the Mather-Jacobian discrepancy
divisor and plays a central role in this paper.
The basic idea is just to replace the usual discrepancy $K_{Y/X}$ by the Mather-Jacobian discrepancy, {\it i.e.,} we define the Mather-Jacobian log discrepancy
$$\ha(E; X, \a^t):=\ord_E(\hk_{Y/X}-J_{Y/X})-t \val_E(\a)+1$$
and by $\ha(E; X, \a^t)\geq 1 $ (resp. $\geq 0$, $> 1$, $>0$) for every exceptional
prime divisor $E$, we define that $(X, \a^t)$ is MJ-canonical (resp. MJ-log canonical,
MJ-terminal, MJ-log terminal. 
We say that $(X, \a^t)$ is MJ-klt if $\ha(E; X, \a^t)> 0 $
for every prime divisor on $Y$. 
According to the basic idea of the replacement by Mather-Jacobian discrepancy, the 
invariants the minimal log discrepancy $\mld$ and the multiplier ideal $\j(X, \a^t)$ defined by using the usual
discrepancy divisor, can be modified to the Mather-Jacobian versions
$\hmld$ and $\hj(X, \a^t)$.

In some points, the Mather-Jacobian discrepancy behaves better  than the usual 
discrepancy divisor.
One of the most distinguished properties of the Mather-Jacobian discrepancy is  the inversion of adjunction:
\begin{prop}[Inversion of Adjunction,  \cite{dd},  \cite{Ishii}]
Let $X$ be a connected reduced equidimenisonal scheme of finite type over $k$.
Let  $A$ be a non singular variety containing $X$ as a closed
subscheme of codimension $c$ and $W$ a strictly  proper closed subset of $X$.
 Let $\widetilde\a\subset \o_A$ be an ideal such that its image ${\a}:=\widetilde\a\o_X\subset \o_X$ is non-zero on each irreducible component.
 Denote the defining ideal of $X$ in $A$ by $I_X$.
Then,
$$\hmld(W; X,{\a^t})=\hmld(W;A,\widetilde\a^t I_X^c)=\mld(W;A,\widetilde\a^t I_X^c).$$
\end{prop}
\noindent
Many good properties follows from this formula.

In this paper we study basic properties of MJ-canonical, MJ-log canonical singularities
and determine these singularities of dimension up to 2.
Concretely we obtain the following.
The first one below is about the relation of singularities of MJ-version and singularities of the usual version.

 \begin{prop}[Proposition \ref{mj-and-usual}]
 Let $X$ be a normal $\bQ$-Gorenstein variety, $\a\subset \ox$ an ideal and $t$ a non negative real number.
 If $(X,\a^t)$ is MJ-canonical (resp. MJ-log canonical, MJ-terminal, MJ-log terminal, MJ-klt), then it is canonical (resp. log canonical, terminal, log terminal, klt) in the usual sense.
 \end{prop}
 
 We call MJ-canonical singularities, MJ-log canonical singularities and so on 
 by the generic name ``MJ-singularities".
As MJ-singularities are not necessarily normal, it is reasonable to compare these 
with existing non normal  singularities which is considered as ``good" singularities.
 The following gives the relation of MJ-log canonical singularities and semi log canonical singularities.
 
 \begin{prop}[Proposition \ref{slc}]
  Assume $X$ is $S_2$ and $\bQ$-Gorenstein.
  If $(X, \a^t)$ is MJ-log canonical, then it is semi log canonical.
\end{prop}

 We sometimes come across  the necessity to compare singularities on two schemes 
 connected by a proper birational morphism.
 The following shows the relation of the Mather-Jacobian discrepancies between the two
schemes:

 \begin{thm}[Theorem \ref{comparison}] 
  Let $\varphi: X'\to X$ be a proper birational morphism which can be extended to 
  a proper birational morphism $\Phi: A'\to A$ of non singular varieties 
  such that  $X'\subset A'$, $X\subset A$  with codimension $c$ and $\Phi$ is isomorphic at the generic point 
  of each irreducible component of $X$.
  Let   $I_X$ and $I_{X'}$ be defining ideals of $X$ and $X'$ in $A$ and $A'$, respectively.
  
  If $I_{X'}{\frak b'}\subset I_X\oaa\subset I_{X'}\frak b$ holds for some ideals $\frak b$, $\frak b'$ in $\oaa$ that do not vanish on any irreducible component of $X'$, then 
  there exists an embedded resolution $\Psi:\ola\to A'$ of $X'$ in $A'$ such that the restriction
  $(\Phi\circ\Psi)|_\olx: \olx \to X$ is a log resolution of $(X, \a\jx)$ and satisfying:
  $$\hk_{\olx/{X'}}-J_{\olx/{X'}}-cR'\leq\hk_{\olx/{X}}-J_{\olx/{X}}-\Psi^*K_{A'/A}\leq \hk_{\olx/{X'}}-J_{\olx/{X'}}-cR,$$
where $R$ and $R'$ are  effective divisors on $\ola$ such that ${\frak b}\ooa=\ooa(-R)$ 
and ${\frak b'}\ooa=\ooa(-R')$.
\end{thm}

By this theorem we obtain many examples of MJ-singularities and it is useful to determine the 2-dimensional MJ-log canonical singularities in \S 5.
We also obtain the relation of MJ-singularities and the  singularities appeared recently
in the paper by De Fernex and Hacon (\cite{dfh}).

\begin{thm}[Theorem \ref{dFH}] Assume that $X$ is normal. If a pair $(X, \a^t)$ is MJ-klt (resp. MJ-log canonical), then it is log terminal
(resp. log canonical) in the sense of De Fernex and Hacon.
\end{thm}
 
 By the property of De Fernex and Hacon's singularities we obtain:
 \begin{cor}[Corollary \ref{usual}] If a pair $(X, \a^t)$ is MJ-klt (resp. MJ-log canonical), 
then there is a boundary $\Delta$ on $X$ such that $((X, \Delta), \a^t)$ is klt 
(resp. log canonical) in the usual sense.
\end{cor}

By the proof  of the above theorem, the relation of MJ-multiplier ideals and De Fernex-Hacon's multiplier ideals. 

\begin{thm}[Theorem \ref{mjmulti}] Let $(X,\a^t)$ be a pair with a normal variety $X$, an ideal $\a$ on $X$ and
$t\in \bR_{\geq 0}$. Then the following inclusion holds for every $m\in \bN$:
$$\hj(X,\a^t)\subset \j_m(X,\a^t),$$
in particular 
$$\hj(X,\a^t)\subset \j(X,\a^t).$$
\end{thm}

It is known that canonical (resp. log canonical) singularities are stable under a 
small flat deformation.
We obtain the similar results for MJ-singularities.
Here, we do not need the flatness of the deformation.
We define that $\{(X_\tau, \a^t_\tau)\}_{\tau\in T}$ is a deformation of $(X_0,\a^t_0)$,
if there is a surjective morphism 
$\pi: X\to T$ 
with equidimensional reduced fibers $X_\tau=\pi^{-1}(\tau)$ of common dimension $r$ for all closed points $\tau\in T$ and there exists an ideal $\a$ on the total space $X$
such that $\a^t_\tau=\a^t\o_{X_\tau}$ are not zero for 
 all $\tau\in T$ .

\begin{thm}[Theorem \ref{deform-log}, \ref{deform-cano}]
  Let $\{(X_\tau, \a^t_\tau)\}_{\tau\in T}$ be a deformation of $(X_0,\a^t_0)$.
  Assume $(X_0,\a^t_0)$ is MJ-canonical (resp. MJ-log canonical) at $x\in X_0$.
  Then there are neighborhoods $X^*\subset X$ of $x$ and $T^*\subset T$ of $0$ such that
  $X^*_\tau$ is MJ-canonical (resp. MJ-log canonical) for every closed point $\tau\in T^*$.
\end{thm}
The lower semi continuity of MJ-minimal log discrepancies is also proved:
\begin{prop}[Proposition \ref{lowsc}]
 Let $\{(X_\tau, \a^t_\tau)\}_{\tau\in T}$ be a deformation of $(X_0,\a^t_0)$ and let 
 $\pi:X\to T$ is the morphism giving the deformation.
 Let $\sigma: T\to X$ a section of $\pi$.
 Then, the map $T\to \bR, \tau \mapsto \hmld(\sigma(\tau), X_\tau, \a^t_\tau)$ is lower 
 semi continuous.
\end{prop}

In the last section we determine all MJ-canonical, MJ-log canonical singularities up
to dimension 2.

\begin{prop}[Proposition \ref{1dim}]
Let $(X,x)$ be a singularity on one-dimensional reduced scheme. Then the following hold:
\begin{enumerate}
\item[(i)] $(X,x)$ is MJ-canonical if and only if   it is non singular.
\item[(ii)] $(X,x)$ is MJ-log canonical if and only if it is non singular or ordinary node.
\end{enumerate}
\end{prop}

\begin{thm}[Theorem \ref{rdp}]
Let $(X,x)$ be a singularity on a 2-dimensional reduced scheme. Then $(X,x)$ is MJ-canonical
if and only if it is non singular or rational double.
\end{thm}

The following theorem gives the total list of 2-dimensional MJ-log canonical singularities.

\begin{thm}[Theorem \ref{hyp3}, \ref{ci}]
Let $(X,0)$ be a singularity on a 2-dimensional reduced scheme with  $\emb(X, 0)=3$.  
Then, 
$(X, 0)$ is an MJ-log canonical singularity  if and only if   $X$ is defined by 
$f(x,y,z)\in k[[x,y,z]]$ as follows:  
\begin{enumerate}
\item[(i)] $\mult_0f=3$ and the projective tangent cone of $X$ at 0 is a reduced curve with at worst ordinary nodes.
\item[(ii)] $\mult_0 f=2$  
\begin{enumerate}
\item $f=x^2+y^2+g(z)$, $\deg g\geq 2$.  
\item $f=x^2+g_3(y,z)+g_4(y,z) $, $\deg g_i\geq i$, $g_3$ is homogeneous of degree 3 and $g_3\neq l^3 $ ($l$ linear)  
\item $f=x^2+y^3+yg(z)+h(z)$, $\mult_0g\leq 4$ or $\mult_0h\leq 6$.  
\item $f=x^2+g(y,z)+h(y,z)$, $g$ is homogeneous of degree 4 and it does not have a linear factor with multiplicity more than 2.  
\end{enumerate}
\end{enumerate}

\noindent 
 Let $(X,0)$ be a singularity on a 2-dimensional reduced scheme with  \\
 $\emb(X, 0)=4$.  
Then, the following hold:
\begin{enumerate}
\item[(iii)] In case 
$(X, 0)$ is locally a complete intersection:

\noindent
$X$ is MJ-log canonical at 0  if and only if \\
$\widehat{\ox_{,0}}\simeq k[[x_1,x_2,x_3,x_4]]/(f,g)$, where $f,g$ satisfy the 
conditions that  
 $\mult_0f=\mult_0 g=2$  and $V(\int(f),\int(g))\subset \bP^3$ 
is a reduced curve with at worst ordinary double points.
\item[(iv)] In case
$(X,0)$ is not locally a complete intersection:

\noindent
$X$ is  MJ-log canonical at 0  if and only if $X$ is a subscheme of
 a locally  complete intersection scheme $M$ which is MJ-log canonical at 0.

\end{enumerate}

\end{thm}

\noindent
{\bf Acknowledgement.}
Main part  of the paper was done during the program Commutative Algebra 
 in 2013 at MSRI.
The authors thank the organizers of the program for the excellent organization and 
thank also MSRI for the hospitality.

\section{Preliminaries}
\noindent
In this paper $X$ is always a connected reduced equidimensional affine scheme of finite type over 
an uncountable algebraically closed field $k$ of characteristic zero. 
Sometimes we put some additional conditions on $X$, but in that case it is always stated clearly. 
Denote the dimension $\dim X=d$.
A variety in this paper always means an irreducible reduced separated scheme of 
finite type over $k$.

Let $\hx \to X$ be the Nash blow-up (for the definition, see for example \cite{DEI}).
The Nash blow-up has the following property:

\noindent
If a resolution $\varphi: Y\to X$ factors through the Nash blow-up $\hx\to X$, the canonical 
homomorphism $\varphi^*(\Omega^d_X) \to \Omega^d_Y$ has the invertible image 
(\cite{DEI}).

\begin{defn} [\cite{DEI}]\label{mather}  Let  $\varphi \colon Y\to X$ be a resolution of singularities of $X$ that factors through the Nash blow-up of $X$. 
By the above comment, the image of the canonical homomorphism
$$\varphi^*(\Omega^d_X) \to \Omega^d_Y$$
is an invertible sheaf of the form $J  \Omega^d_Y$, 
where $J$ is the invertible ideal sheaf on $Y$ that defines an effective divisor supported on the exceptional locus of $\varphi$. 
This divisor is called the \emph{Mather discrepancy divisor} and denoted by $\hKY$.
\end{defn}

\begin{defn}
Recall that the \emph{Jacobian ideal} $\j_X$ of a variety $X$ is the $d^{\rm th}$ Fitting ideal 
${\rm Fitt}_d(\Omega_X)$ of $\Omega_X$. 
If $\varphi :Y\to X$ is a log resolution of $\j_X$, we denote by $J_{Y/X}$ the effective divisor
on $Y$ such that  $\j_X\o_Y=\o_Y(-J_{Y/X})$. 
This divisor is called the \emph{Jacobian discrepancy divisor}. 
\end{defn}

Here, we note that every log-resolution of $\j_X$ factors through the Nash blow-up (\cite[Remark 2.3]{eim}).

\begin{defn}\label{mjdiscrep}
Let $\a\subseteq \o_X$ be a nonzero ideal on $X$, and 
 $t\in{\bR}_{\geq 0}$. Given a log resolution $\varphi :Y\to X$ of $\j_X \a$,
we denote by $Z_{Y/X}$ the effective divisor on $Y$ such that
 $\a\o_Y=\o_Y(-Z_{Y/X})$. 
 For a prime divisor $E$ over $X$, we define the Mather-Jacobian-log discrepancy
 (MJ-log discrepancy for short) at $E$ as
 $$\ha(E; X, \a^t):=\ord_E(\hk_{Y/X}-J_{Y/X}- t Z_{Y/X})+1.$$

\end{defn}

\begin{rem}
For nonzero ideals $\a_1,\ldots,\a_r$ on  $X$, one can similarly define a mixed
MJ-log discrepancy $\ha (E; X, \a_1^{t_1}\cdots\a_r^{t_r})$ for every $t_1,\ldots,t_r\in{\bR}_{\geq 0}$. With the notation in Definition~\ref{mjdiscrep}, if 
$f$ is a log resolution of ${\j }_X\a_1\cdots \a_r$, and if we put
$\a_i\cO_Y=\cO_Y(-Z_i)$, then
$$\ha (E; X, \a_1^{t_1}\cdots\a_r^{t_r})
=\ord_E(\hk_{Y/X}-J_{Y/X}- t_1Z_1-\ldots-t_rZ_r)+1.$$
For simplicity, we will mostly state the results for a pair $(X, \a^t)$ with one  ideal, but all statements
have obvious generalizations to the mixed case.
\end{rem}
\begin{rem}\label{note-on-a}
If $X$ is normal and locally a complete intersection, then $\ha (E;X, \a^t)=a (E;X, \a^t)$, where the right hand side is the usual log discrepancy $\ord_E(K_{Y/X}-t Z_{Y/X})+1$.
Indeed, in this case the image of the canonical map $\Omega^d_X \to \omega_X$ is $\j_X\omega_X$,
hence 
$\hk_{Y/X}-J_{Y/X}=K_{Y/X}$.
In particular, we see that $\ha (E; X, \a^t)=a(E; X, \a^t)$ if $X$ is smooth.
\end{rem}

\begin{defn} Let $X$ be a  normal and $\bQ$-Gorenstein variety.
Let $W$ be a proper closed subset of $X$.
The {\it minimal log-discrepancy} of $(X,\a^t)$ along $W$ is defined as follows:

\noindent
If $\dim X\geq 2$,
$$\mld(W; X,\a^t)=\inf \{a(E; X, \a^t) \mid E\ \mbox{prime\ divisor\ over}\ X \ \mbox{with\ center\ in\ } W\}.$$
When $\dim X=1$, we use the same definition as above, unless the infimum is negative,
in which case we make the convention that $\mld (W;X,\a^t)=-\infty$.
\end{defn}

Now returning to the general setting on $X$, we  define a modified invariant.

\begin{defn}
Let $W$ be a closed subset of $X$ such that it does not contain an irreducible component of $X$.
(We call such a closed subset a ``strictly  proper closed subset" in this paper.)
Let $\eta$ be a point of $X$ such that its closure is a strictly  proper closed subset of $X$.
The {\it Mather-Jacobian minimal log-discrepancy} of $(X,\a^t)$ along $W$ (resp. at $\eta$) are defined as follows:

\noindent
If $\dim X\geq 2$,
$$\hmld(W; X,\a^t)=\inf \{\ \ha(E;X,\a^t) \mid E\ \mbox{prime\ divisor\ over}\ X
\ \mbox{with\ center\ in\ }W\}.$$
$$\hmld(\eta; X,\a^t)=\inf \{\ \ha(E;X,\a^t) \mid E\ \mbox{prime\ divisor\ over}\ X
\ \mbox{with\ center\  }\overline{\{\eta\}}\}.$$

(Note that we strictly distinguish between ``center in $Z$" and ``center $Z$".)

When $\dim X=1$, we use the same definition as above, unless the infimum is negative,
in which case we make the convention that $\hmld (W;X,\a^t)=-\infty$ (resp. $\hmld (\eta;X,\a^t)=-\infty$ ).
\end{defn}

\begin{rem}\label{conflict}
\begin{enumerate}
\item[(i)]
By Remark \ref{note-on-a}, we have $$\mld(W;X,\a^t)=\hmld (W;X,\a^t),$$ if $X$ is
normal and locally a 
complete intersection.
\item[(ii)]
In case $\dim X\geq 2$, if there is a prime divisor $E$ with the center in $W$ such that
$\ha(E;X,\a^t)<0$, then $\hmld(W; X,\a^t)=-\infty$. 
This is proved by using $\hk_{Y'/X}-J_{Y'/X}=K_{Y'/Y}+\psi^*(K_{Y/X}-J_{Y/X})$
for another resolution $Y'\to X$ factoring through $Y\to X$, 
in the similar way as the usual discrepancy case.
\item[(iii)]
There are some conflicts of notation in \cite{dd},  \cite{eim}, \cite{Ishii} and \cite{ir}, since these papers are working on the same materials and some of these papers were done independently of  others.
Here, we propose the notation $\hmld (W;X,\a^t)$ for Mather-Jacobian minimal log discrepancy, while in \cite{dd} it is denoted as $\mld^\diamond(W;X,\a^t)$ and in \cite{ir}
as $\widehat\mld (W;X, \j_X\a^t)$. 
We hope the new notation here is appropriate to unify the notation.

\end{enumerate}
\end{rem}

\begin{prop}[Inversion of Adjunction \cite{dd},\cite{Ishii}]\label{inversion}
Let  $A$ be a non singular variety containing $X$ as a closed
subscheme of codimension $c$ and $W$ a strictly  proper closed subset of $X$.
 Let $\widetilde\a\subset \o_A$ be an ideal such that its image ${\a}:=\widetilde\a\o_X\subset \o_X$ is non-zero on each irreducible component of $X$.
 Denote the defining ideal of $X$ in $A$ by $I_X$.
Then,
$$\hmld(W; X,{\a^t})=\hmld(W;A,\widetilde\a^t I_X^c)=\mld(W;A,\widetilde\a^t I_X^c).$$
\end{prop}
Here, the second equality is trivial by Remark \ref{conflict}, (1).
The Inversion of Adjunction is proved by discussions of jet schemes and we also 
use them in this paper.
Here, we introduce the basic notion of jet schemes.

\begin{defn}
 Let $K\supset k$ be a field extension and $m\in \bZ_{\geq 0}$.
  A  morphism \( \spec K[t]/(t^{m+1})\to X \) is called an \( m \)-jet
  of \( X \) and \( \spec K[[t]]\to X \) is called an {\it arc} of \( X \).
\end{defn}

\begin{say}
\label{field}
  Let \( {\mathcal S}ch/k \) be the category of \( k \)-schemes  
   and \( {\mathcal S}et \) the category of sets.
  Define a contravariant functor  \( F_{m}: {\mathcal S}ch/k \to {\mathcal S}et \)
  by 
$$
 F_{m}(Y)=\Hom _{k}(Y\times_{\spec k}\spec k[t]/(t^{m+1}), X).
$$
  Then, \( F_{m} \) is representable by a scheme \( \lm(X) \) of finite
  type over \( k \), {\it i.e.,} $\lm(X)$ is the fine moduli scheme of $m$-jets of $X$.
  
  The scheme \( \lm(X) \) is  called the {\it scheme of \( m \)-jets} of \( X \).
  
   In the same way, the fine moduli scheme $\li(X)$ of arcs of $X$ also exists and
   it is called the {\it scheme of arcs} of $X$.
   We should note that $\li(X)$ is not necessarily of finite type over $k$. 
The canonical surjection \( k[t]/(t^{m+1})\to k[t]/(t^{n+1}) \) $(n<m\leq \infty)$
  induces a morphism \( \psi_{mn}:\lm(X)\to \ln(X) \).
  
  If $X=\spec k[x_1,\ldots, x_N]/(f_1,\ldots, f_r)$, 
  then $$\lm(X)=\spec k[\bx^{(0)},\bx^{(1)},\ldots, \bx^{(m)}]/(F_i^{(j)})_{1\leq i\leq r, 0\leq j\leq m},$$
  where $\bx^{(j)}=(x_1^{(j)},\ldots, x_N^{(j)})$ and 
  $ \sum_{j=0}^\infty F^{(j)}t^j$ is the Taylor
expansion of $f(\sum_j {\bx}^{(j)} t^j)$, hence $F^{(j)} \in
k[{\bx^{(0)}}, \ldots, \bx^{(j)}]$.
If $0\in X\subset \bA^N$, we have 
\begin{equation}\label{fiber}
\psi_{m0}^{-1}(0)=\spec k[\bx^{(1)},\ldots, \bx^{(m)}]/(\overline{F}_i^{(j)})_{1\leq i\leq r, 0\leq j\leq m},
\end{equation}
where $\overline{F}_i^{(j)}$ is the image of ${F}_i^{(j)}$ by the canonical projection map \\
$k[\bx^{(0)},\bx^{(1)},\ldots, \bx^{(m)}]\to k[\bx^{(1)},\ldots, \bx^{(m)}]$ which sends $\bx^{(0)}$ to $0$.
\end{say}

\begin{rem}\label{keisan}
  Under the notation above, for a polynomial $f\in k[x_1,\ldots, x_N]$, let 
  $$f\left(\sum x_1^{(j)}t^j, \ldots, \sum x_N^{(j)}t^j\right)=F^{(0)}+F^{(1)}t+ F^{(2)}t^2+\cdots $$
  be the Taylor expansion.
  Then a monomial in $F^{(j)}$ is of the type $$ x_{i_1}^{(e_1)}\cdots x_{i_r}^{(e_r)} \ \
  (e_l\geq 0, i_l \in \{1,\ldots, N\},\ \  \sum_l e_l=j).$$
  Here, if $r>j$, then the monomial must contain a factor $x_{i_l}^{(0)}$, therefore  the image of this monomial by the projection map $k[\bx^{(0)},\bx^{(1)},\ldots, \bx^{(m)}]\to k[\bx^{(1)},\ldots, \bx^{(m)}]$ is zero.
  By this observation we obtain that if $j<\mult_0 f$, then $\overline F^{(j)}=0$ and if $j=\mult_0 f$, 
  then $\overline F^{(j)}=\int f(\bx^{(1)})$, where $\int f$ is the initial term of $f$ with the usual 
  grading in $k[x_1,\ldots, x_N]$.

\end{rem}

By the Inversion of Adjunction,  we can describe Mather-Jacobian discrepancy in terms  of the jet schemes of $A$ as follows:

\begin{prop} \label{description}
Let $X$, $A$, $c$,  $\a$ and $\widetilde\a$ be as  in Proposition \ref{inversion}. Let $N=d+c$ and $Z=V(\widetilde\a)$.
Let $\psi_m: \mathcal{L}^\infty(A)  
\to \lm(A)$ and $\psi_{m,n}:\lm(A)\to \ln(A)$ be the canonical projections of jet schemes of $A$.
Then, 
 
$$\hmld(W; X,{\a^t})=\inf_{m,n\in \bZ_{\geq 0}}\{(M+1)N-(m+1)t-(n+1)c $$
$$\ \ \ \ \ \ \ \ \ \ \ \ \ \ \ \ \ \ \ \ - \dim\left(\psi_{Mm}^{-1}(\lm(Z))
\cap \psi_{Mn}^{-1}(\ln(X))\cap \psi_{M0}^{-1}(W)\right)\},$$
where $M=\max\{m,n\}$.

In particular for $\a^t=\ox$ we obtain:
\begin{equation}\label{hmld} \hmld(W; X, \ox)=\inf_{n\in \bZ_{\geq 0}}\{(n+1)d-\dim(\psi^X_{n0})^{-1}(W)\},
\end{equation}
where $\psi^X_{n0}: \lm(X)\to \mathcal{L}^0( X)=X$ is the canonical projection of jet schemes of $X$.
\end{prop}

\begin{proof}
By the Inversion of Adjunction, we can represent
$$\hmld(W; X,{\a^t})=\hmld(W;A,\widetilde\a^t I_X^c)=\mld(W;A,\widetilde\a^t I_X^c).$$
By \cite[Remark 3.8]{Ishii}, this is represented as 
 $$\mld(W;A,\widetilde\a^t I_X^c)=\inf_{m,n\in\bN}\{\codim(\cont^{\geq m}(\a)\cap
 \cont^{\geq n}(I_X)\cap\cont^{\geq 1}(I_W))-mt-nc\},$$
 where $\codim$ is the codimension in the arc space $\li(A)$.
 By shifting $m$ to $m+1$ and $n$ to $n+1$, this is represented as 
 $$\inf_{m,n\in\bZ_{\geq 0}}\{\codim(\cont^{\geq {m+1}}(\a)\cap
 \cont^{\geq n+1}(I_X)\cap\cont^{\geq 1}(I_W))-(m+1)t-(n+1)c\},$$
 Now noting that  $$\operatorname{Cont}^{\geq m+1}(\a)=\psi_{m}^{-1}(\lm(Z))\ \mbox{and}$$
 $$\operatorname{Cont}^{\geq n+1}(I_X)=\psi_{n}^{-1}(\ln(X)),$$
 we obtain the equality 
 $$\codim(\cont^{\geq {m+1}}(\a)\cap
 \cont^{\geq n+1}(I_X)\cap\cont^{\geq 1}(I_W))$$
$$
=\codim\left(\psi_{Mm}^{-1}(\lm(Z))
\cap \psi_{Mn}^{-1}(\ln(X))\cap \psi_{M0}^{-1}(W), \ \ {\mathcal L}_M(A) \right),$$
where $M=\max \{m,n\}$.
As $\dim A=N$, we have $\dim {\mathcal L}_M(A)=(M+1)N$ which yields the 
required equality.
\end{proof}

Now we define an exceptional divisor over $X$, which is a generalization of an exceptional
divisor for normal variety (Note that  if $X$ is normal, an exceptional divisor is  defined as a divisor over $X$ with the center of codimension $\geq 2$ on $X$.)
\begin{defn}
Let $E$ be a prime divisor over $X$.
Let $\varphi:Y\to X$ be a proper birational morphism such that $Y$ is normal and $E$ appears 
on $Y$.
Then $E$ is called an exceptional divisor over $X$ if $\varphi$ is not isomorphic at the generic point of $E$.
Here, we note that this definition is independent of the choice of $\varphi$.
\end{defn}

\begin{defn} \label{defmj}
   We call a pair $(X, \a^t)$ consisting of a connected reduced equidimensional scheme $X$ of finite type over $k$ and an ideal $\a\subset \ox$ with a    
   non negative real number $t$ is {\it MJ-canonical} (resp. {\it MJ-log canonical})
   if for every exceptional prime divisor $E$ over $X$, the inequality $\ha(E; X,\a^t)\geq 1$
   (resp. $\geq 0$) holds.
   
   We say that $(X, \a^t)$ is MJ-canonical (resp. MJ-log canonical) at a point $x\in X$, if 
   there is an open neighborhood $U\subset X$ of $x$ such that $(U, \a^t|_U)$ is
   MJ-canonical (resp. MJ-log canonical). 
   
   If $(X,\ox)$ is MJ-canonical (resp. MJ-log canonical), we say that $X$ is MJ-canonical
   (resp. MJ-log canonical), or $X$ has MJ-canonical (resp. MJ-log canonical) singularities.
   
      In the similar way, we can define {\it MJ-terminal} and {\it MJ-log terminal} by the conditions for all exceptional prime divisors.
      In addition, we say that $(X, \a^t)$ is {\it MJ-klt} if for every prime divisor $E$ over
      $X$, the  inequality $\ha(E; X,\a^t)> 0$ holds.
\end{defn}


\begin{defn}
  Let $(X, \a^t)$ be a pair consisting of $X$ and an ideal $\a\subset \ox$ with a    
   non negative real number $t$.
  Let $\varphi: Y\to X$ be a log resolution of $(X, \a\jx)$. Define a divisor $Z_{Y/X}$
  by $\oy(-Z_{Y/X})=\a\oy$.
  Then we can define the {\it Mather-Jacobian multiplier ideal} (or {\it MJ-multiplier ideal} for short) as follows:
   $$\hj(X, \a^t)=\varphi_*(\oy(\hk_{Y/X}-J_{Y/X}-[tZ_{Y/X}])),$$
   where $[  D ]$ is the round down of the real divisor $D$. 

\end{defn}

\begin{rem} At the stage of the definition, this multiplier ``ideal" is only a fractional
ideal for non normal $X$. But in \cite{eim} we proved that it is really an ideal of $\ox$ in general.
In \cite{eim}, the MJ-multiplier ideal is proved to have good properties which a ``multiplier ideal" is expected to have.

In \cite{eim} this multiplier ideal is called Mather multiplier ideal and denoted 
by $\widehat{\j}(X,\cdots)$.
On the other hand, in \cite{dd} MJ-canonical (resp. MJ-log canonical) are called 
J-canonical (resp. log J-canonical).
Here we think that it is more appropriate to call these notions with both M and J.
 \end{rem}
 
 \begin{rem}
 Fix a log resolution $Y\to X$  of $(X, \jx\a)$. 
 Then $(X, \a^t)$ is MJ-canonical (resp.
 MJ-log canonical, MJ-terminal, MJ-log terminal) if and only if $\ha(E; X,\a^t)\geq 1$ (resp. $\geq 0$, $>1$, $>0$) for all 
 exceptional prime divisor $E$ on $Y$. 
 Also $(X, \a^t)$ is MJ-klt if and only if $\ha(E; X,\a^t)> 0$ for every prime divisor $E$ on $Y$.
 This is proved by using the fact that 
 $$\hk_{Y'/X}-J_{Y'/X}=K_{Y'/Y}+\psi^*(K_{Y/X}-J_{Y/X})$$
for another resolution $Y'\to X$ factoring through $Y\to X$.
 \end{rem}
 \begin{rem}
 
 Assume that  $X$ is  normal and locally a complete intersection.
   Then by Remark \ref{note-on-a},
    MJ-canonical (resp. MJ-log canonical) are equivalent to canonical (resp. log canonical).
 For normal and $\bQ$-Gorenstein case, we have the following:  
 
 \end{rem}
 \begin{prop}\label{mj-and-usual}
 Let $X$ be a normal $\bQ$-Gorenstein variety, $\a\subset \ox$ an ideal and $t$ a non negative real number.
If $(X,\a^t)$ is MJ-canonical (resp. MJ-log canonical, MJ-terminal, MJ-log terminal, MJ-klt), then it is canonical (resp. log canonical, terminal, log terminal, klt) in the usual sense.
 \end{prop}
 
 \begin{proof} Let the index of $X$ be $r$, then the image of the canonical map
 $$(\wedge^d\Omega_X)^{\otimes r}\to \omega_X^{[r]}$$
 is written as $I_r \omega_X^{[r]}$ with an ideal $I_r$
since  $\omega_X^{[r]}$ is invertible.
Then, by the definition of the Mather discrepancy and the usual discrepancy,
we have $$I_r\oy(r\hk_{Y/X})=\oy(rK_{Y/X})$$
for a log resolution $Y\to X$ of $(X, \jx \a)$.
Let $J_r=\jx^r:I_r$, then ${J_r I_r}$ and $\jx^r$ have the same integral closures
by \cite[Corollary 9.4]{e-Mus2}. 
Therefore if we write $\oy(-Z_r)=I_r\oy $ and $\oy(-Z'_r)=J_r\oy$, then $rJ_{Y/X}=
Z_r+Z'_r$ and 
$$r\hk_{Y/X}-rJ_{Y/X}=r\hk_{Y/X}-Z_r-Z'_r=rK_{Y/X}-Z'_r\leq rK_{Y/X},$$
which gives our assertions.
  \end{proof}

\begin{prop}\label{atx}
\begin{enumerate}
  \item[(i)] A pair $(X, \a^t)$ is MJ-log canonical at a (not necessarily closed) point $x\in X$ 
  if and only if $$\hmld(x; X, \a^t)\geq 0.$$
  \item[(ii)] If a pair $(X, \a^t)$ is  MJ-canonical at a (not necessarily closed) point $x\in X$ then 
  $$\hmld(x; X, \a^t)\geq 1.$$
\end{enumerate}
\end{prop}  

\begin{proof} It is clear that if a pair $(X, \a^t)$ is MJ-log canonical (resp. MJ-canonical) 
at a point $x\in X$ then $\hmld(x; X, \a^t)\geq 0$
(resp. $\hmld(x; X, \a^t)\geq 1$) by the definitions.
For the proof of the converse statement in (i), we have only to note that 
$$\hk_{Y'/X}=K_{Y'/Y}+\varphi^*\hk_{Y/X}$$
for another resolution $Y'$ of $X$  that dominates $Y$ by $\varphi:Y'\to Y$.
The proof of the proposition is the same as the corresponding statement for the usual minimal log discrepancy.

\end{proof}

  The converse of the statement of (ii) in Proposition \ref{atx} does not hold.
  The following is an example for that. 
\begin{exmp}
  Let X be a hypersurface in $\bA^3$ defined by $x_1x_2=0$, where $x_1,x_2,x_3$ are the coordinates of $\bA^3$. Then the $x_3$-axis $C$ is the singular locus of $X$.
By the Inversion of Adjunction, we have $\hmld(C; X, \ox)=\mld (C; \bA^3, (x_1x_2))$, where the
  right hand side is known to be zero. Therefore $X$ is not MJ-canonical at the origin $0$.
  On the other hand, again by the Inversion of Adjunction,
  $\hmld(0;X, \ox)=\mld(0;  \bA^3, (x_1x_2))$, where the right hand side is known to be 1.
\end{exmp}
 
 In the definition \ref{defmj} of MJ-log canonical singularities, the conditions are for exceptional 
 prime divisors over $X$.
 But we can replace them by prime divisors over $X$.
 
 \begin{prop}
 A pair $(X, \a^t)$ is MJ-log canonical if and only if \\
 $\ha(E; X, \a^t)\geq 0$ holds for every
 prime divisor $E$ over $X$.
 \end{prop}
 
 \begin{proof}
  The ``if" part of the proof is obvious. 
  For the converse, 
  we have only to note that 
$$\hk_{Y'/X}=K_{Y'/Y}+\varphi^*\hk_{Y/X}$$
for another resolution $Y'$ of $X$  that dominates $Y$ by $\varphi:Y'\to Y$.
The proof of the statement is the same as the corresponding statement for the usual log discrepancy.  
 \end{proof}

\section{Basic properties of the MJ-singularities}
In this section, we show some basic properties on MJ-singularities.

\begin{prop}[\cite{dd}, \cite{eim}]\label{cano=normal}
If $X$ is MJ-canonical, then it is normal and has rational singularities.
\end{prop}

\begin{prop}[\cite{dd}] 
If $k=\bC$ and $X$ is MJ-log canonical, then $X$ has Du Bois singularities.
\end{prop}

We will see that the class of  Du Bois singularities is much wider than that of MJ-log canonical singularities (see Example \ref{duboiscurve}).

\begin{prop}\label{emb}
  Let $x\in X$ be a closed point.
  If $X$ is MJ-canonical at $x$, then the embedding dimension $\emb(X,x)\leq 2d-1$.
  If $X$ is MJ-log canonical at $x$, then the embedding dimension $\emb(X,x)\leq 2d$.
\end{prop}

\begin{proof}
  By (\ref{hmld}) in Proposition \ref{description}, with putting $W=\{x\}$ we have 
  $$\hmld(x; X, \ox)=\inf_{n\in \bZ_{\geq 0}}\{(n+1)d-\dim(\psi^X_{n0})^{-1}(x)\}.$$
  If $X$ is MJ-canonical at $x$, then $\hmld(x; X, \ox)\geq 1$ and this implies that
  $\dim(\psi^X_{n0})^{-1}(x)\leq (n+1)d-1$ holds for every $n\in \bN$.
  Therefore, in particular for $n=1$, we have 
  $$\dim (T_{X,x})=\dim( \psi^X_{1,0})^{-1}(x)\leq 2d-1,$$
  where $T_{X,x}$ is the Zariski tangent space of $X$ at $x$. Hence the embedding dimension of $X$ at $x$ is $2d-1$.
  The proof for the statement on MJ-log canonical  singularities follows in the same way.
\end{proof}

\begin{defn}
  Let $X$ be embedded in a non singular variety $A$ and $I_X$ the defining ideal of $X$
  in $A$.
  Let $\Phi:\overline{A} \to A$ be a proper birational morphism which is isomorphic on the 
  generic point of each irreducible component of $X$.
  Let $\olx$ be the strict transform of $X$ in $\ola$ and $I_\olx$ be the defining ideal of 
  $\olx$ in $\ola$.
  Then, we call $\Phi$ a factorizing resolution of $X$ in $A$ if the following hold:
 \begin{enumerate}
  \item[(i)] $\Phi$ is an embedded resolution of $X$ in $A$;
  \item[(ii)] There is an effective divisor $R$ on $\ola$ such that
    $$I_X\ooa=I_\olx\ooa(-R).$$
 \end{enumerate}
\end{defn}

The existence of factorizing resolution of a given embedding $X\subset A$ is proved by 
A. Bravo and O. Villamayor (\cite{bv}) and  E. Eisenstein obtained in \cite{ee} a modified version 
which can be applied to the case of log resolutions. 
The following is an easy corollary of   \cite[Lemma 3.1]{ee}.

\begin{prop}\label{eefact}
Let $X\subset A$ be a closed embedding  into a non singular variety $A$ and let $\a$ and $\b$ be
 ideals of $\ox$ and $\o_A$, respectively. Let $\widetilde\a$ be an ideal such that 
 $\tilde \a\ox=\a$.
 Assume that $\a$ and $\b$ are not zero on the generic point of each irreducible component of $X$.
Then, there exists a factorizing resolution
$\Phi:\ola\to A$ of $X$ in $A$ such that $\Phi$ is a log resolution of $(A, \widetilde\a\b)$ and  the restriction $\Phi|_\olx$ of $\Phi$ onto the strict transform 
$\olx$ is  a log resolution of $(X, \a)$.

\end{prop}

We sometimes come across  the situation to compare the MJ-discrepancies of two
schemes  connected by  a proper birational morphism.
The following gives some information on that.

\begin{thm}\label{comparison}
  Let $\varphi: X'\to X$ be a proper birational morphism which can be extended to 
  a proper birational morphism $\Phi: A'\to A$ of non singular varieties 
  such that  $X'\subset A'$, $X\subset A$  with codimension $c$ and $\Phi$ is isomorphic at the generic point 
  of each irreducible component of $X$.
  Let   $I_X$ and $I_{X'}$ be defining ideals of $X$ and $X'$ in $A$ and $A'$, respectively.
  
  If $I_{X'}{\frak b'}\subset I_X\oaa\subset I_{X'}\frak b$ holds for some ideals $\frak b$, $\frak b'$ in $\oaa$ that do not vanish on any irreducible component of $X'$, then 
  there exists an embedded resolution $\Psi:\ola\to A'$ of $X'$ in $A'$ such that the restriction
  $(\Phi\circ\Psi)|_\olx: \olx \to X$ is a log resolution of $(X, \a\jx)$ and satisfying:
  $$\hk_{\olx/{X'}}-J_{\olx/{X'}}-cR'\leq\hk_{\olx/{X}}-J_{\olx/{X}}-\Psi^*K_{A'/A}\leq \hk_{\olx/{X'}}-J_{\olx/{X'}}-cR,$$
where $R$ and $R'$ are  effective divisors on $\ola$ such that ${\frak b}\ooa=\ooa(-R)$ 
and ${\frak b'}\ooa=\ooa(-R')$.
\end{thm}

For the proof of the proposition, we need the following lemma:

\begin{lem}\label{lemcom} 
  Let $X$ be embedded into a non singular variety $A$ with codimension $c$, 
   $\Phi:\ola \to A$  a proper birational morphism of non singular varieties isomorphic at the generic points of the irreducible components of $X$ and $\olx$ the strict transform of $X$ in $\ola$.
   Denote the ideal of $X$ and $\olx$ by $I_X$ and $I_\olx$, respectively.
   Assume 
 \begin{equation} \label{inclusion} I_\olx\ooa(-R')\subset I_X\ooa\subset I_\olx\ooa(-R),
 \end{equation}
   for some effective divisors $R,R'$ on $\ola$ that do not 
contain any irreducible component of $\olx$ in their supports.
   Then, we have
 \begin{equation} \label{kekka} (K_{\ola/A}-cR')|_{\olx}\leq \hk_{\olx/X}-J_{\olx/X}\leq (K_{\ola/A}-cR')|_{\olx}.
 \end{equation}
 In particular, if $I_X\o_\ola=I_\olx\ooa(-R)$, then
 $$\hk_{\olx/X}-J_{\olx/X}= (K_{\ola/A}-cR)|_{\olx}.$$

\end{lem}

\begin{proof} We use the notation in \cite{ee}. 
The  notation $\left[a_{ij}\right]_c$ means the ideal generated by
$c$-minors of the matrix $(a_{ij})$.
Now since the problem is local, it is sufficient to show  the statement at a neighborhood of a point $P\in\ola$.
Let $I_X$ be generated by $h_1,\ldots, h_m$ around $\Phi(P)$.
Let $(z_1,\ldots,z_N)$ be local coordinates of $A$ at $\Phi(P)$ and 
$(w_1,\ldots, w_d,w_{d+1},\ldots, w_N)$ local coordinates of $\ola$ at $P$ such that
$(w_1,\ldots, w_d)$ is local coordinates of $\olx$.
Then, by \cite[Lemma 4.3]{ee}, it follows:
\begin{equation}\label{eeformula}
\o_\olx(-\hk_{\olx/X})\left(\left[\frac{\partial(h_i\circ\Phi)}{\partial w_j}\right]_c\right)\o_\olx
=\ooa(-K_{\ola/A})\left(\left[\frac{\partial h_i}{\partial z_j}\right]_c\ooa\right)\o_\olx,
\end{equation}
where the right hand side coincides with 
$$\o_\olx(-K_{\ola/A}|_\olx-J_{\olx/X}).$$
Let $g$ and $g'$ be local generators of $\ooa(-R)$ and $\ooa(-R')$ at $P$, respectively.
As $I_\olx$ is generated by $w_{d+1},\ldots, w_N$, the condition of the lemma implies:
$$(g'w_{d+1},\ldots, g'w_N)\subset I_X\ooa=(h_1\circ \Phi,\ldots, h_m\circ \Phi)\subset
(gw_{d+1},\ldots, gw_N).$$
Then, we obtain:
\begin{equation}\label{3jacobian}
\left[\frac{\partial(g'w_i)}{\partial w_j}\right]_c|_\olx\subset
\left[\frac{\partial(h_i\circ \Phi)}{\partial w_j}\right]_c|_\olx\subset
\left[\frac{\partial(gw_i)}{\partial w_j}\right]_c|_\olx.
\end{equation}
Here, we used a general fact: If $I=(g_1,\ldots, g_n)\subset J=(f_1,\ldots, f_m)$ are 
ideals. Then for a closed subscheme $Z\subset Z(J)$, it holds that 
$$\left[\frac{\partial g_i}{\partial w_j}\right]_c|_Z\subset \left[\frac{\partial f_i }{\partial w_j}\right]_c|_Z.$$

Note that $\frac{\partial(gw_i)}{\partial w_j}|_\olx=\left(g\frac{\partial w_i}{\partial w_j}+w_i
\frac{\partial g}{\partial w_j}\right)|_\olx=g\frac{\partial w_i}{\partial w_j}$,
since $w_i=0$ on $\olx$ for $i=d+1,\ldots, N$.
Here, we obtain 
$$\left[\frac{\partial(gw_i)}{\partial w_j}\right]_c|_\olx=g^c|_\olx,$$
and similarly
$$\left[\frac{\partial(g'w_i)}{\partial w_j}\right]_c|_\olx=g'^c|_\olx.$$
Therefore, the inclusions of (\ref{3jacobian}) turn out to be
$$({g'}^c)|_\olx\subset \left[\frac{\partial(h_i\circ \Phi)}{\partial w_j}\right]_c|_\olx
\subset (g^c)|_\olx.$$
Substituting this into (\ref{eeformula}) we obtain
$$\o_\olx(-\hk_{\olx/X}-cR')\subset \o_\olx(-K_{\ola/A}-J_{\olx/X})\subset\o_\olx(-\hk_{\olx/X}-cR),$$
which proves the required inequalities.
\end{proof}

\noindent
{\it Proof of Theorem \ref{comparison}}. 
Applying Proposition \ref{eefact} to $X'\subset A'$, we obtain a factorizing resolution 
$\Psi:\ola\to A'$ of $X'$ in $A'$, such that it is a log resolution of 
$(A',\b\b'\widetilde\jx\widetilde\j_{X'})$, where $\widetilde\jx$ and $\widetilde\j_{X'}$ are
  ideals of $\o_{A'}$
such that $\widetilde\jx\o_{X'}=\jx\o_{X'}$ and $\widetilde\j_{X'}\o_{X'}=\j_{X'}$, respectively.
Let $\b\ooa=\ooa(-R)$ and  $\b'\ooa=\ooa(-R')$.
As $\Psi$ is a factorizing resolution of $X'$ in $A'$, there exists an effective divisor 
$G$ on $\ola$ such that
$$I_{X'}\ooa=I_{\olx}\ooa(-G).$$
By the assumption of the proposition,
we have
$$I_{X'}\ooa(-R')\subset I_X\ooa=(I_X\o_{A'})\ooa\subset I_{X'}\ooa(-R),$$
which yields
$$I_\olx\ooa(-G-R')\subset I_X\ooa\subset I_\olx\ooa(-G-R).$$
Now by Lemma \ref{lemcom}, we obtain
$$(K_{\ola/A}-cG-cR')|_{\olx}\leq \hk_{\olx/X}-J_{\olx/X}\leq (K_{\ola/A}-cG-cR')|_{\olx}.$$
By substituting $K_{\ola/A}=K_{\ola/{A'}}+\Psi^*K_{A'/A}$ and 
$(K_{\ola/{A'}}-cG)|_{\olx}= \hk_{\olx/{X'}}-J_{\olx/{X'}}$ which follows from the second statement 
of Lemma \ref{lemcom}, we conclude the inequalities:
 $$\hk_{\olx/{X'}}-J_{\olx/{X'}}-cR'\leq\hk_{\olx/{X}}-J_{\olx/{X}}-\Psi^*K_{A'/A}\leq \hk_{\olx/{X'}}-J_{\olx/{X'}}-cR.$$
$\Box$

\begin{rem}\label{linear} Let us make a comment about a condition of Theorem\ref{comparison}.
  Locally on $X$, every projective birational morphism $X'\to X$ can be extended to a projective 
  birational morphism $A'\to A$ of non singular varieties.
  This is proved as follows.
  We can assume that $X$ is embedded in $\bA^N$ and $X'\to X$ is a blow-up by an ideal $\mathcal I=({f_1},\ldots,{f_r})$ of $\ox$.  Extend the canonical surjective homomorphism  $k[ x_1,\ldots, x_N]\to \Gamma(X,\ox)$ to 
   a  homomorphism $k[ x_1,\ldots, x_N,y_1,\ldots, y_r]\to \Gamma(X, \ox)$ by $y_i\mapsto f_i$
  for $i=1,\ldots, r$.
  Let $X\subset \bA^{N+r}$ be the embedding corresponding to this homomorphism.
  Then the blow-up $\Phi:A'\to A$ by the ideal $(y_1,\ldots, y_r)$ gives the blow-up 
  by the ideal $\mathcal I$ on $X$. 
  Since the center of the blow-up $\Phi$ is non singular, $A'$ is also non singular.
  
  The most effective application of Theorem \ref{comparison} is for the case that 
  $X'\to X$ is the blow-up at a closed point.
\end{rem}

\begin{cor}\label{blowup}
 Let $X\subset A$ be a closed embedding into a non singular variety $A$ with codimension $c$ and
$\a$ an ideal of $\ox$.
Let $\Phi:A'\to A$ be the blow-up of $A$ at a closed point $x\in X$ and $X'$ the strict 
transform of $X$.
Let $E$ be the exceptional divisor for $\Phi$ and non negative integers $a,b$ as 
$$I_{X'}\o_{A'}(-aE)\subset I_X\o_{A'}\subset I_{X'}\o_{A'}(-bE).$$
Then, there is a proper birational morphism $\Psi:\ola\to A'$ with the strict transform $\olx$
of $X$ in $\ola$ such that the restriction $\Phi\circ\Psi|_{\olx}:\olx\to X$ is a log resolution
of $(X,\a\jx)$ and $\Psi|_\olx:\olx\to X'$ is a log resolution of $(X', \j_{X'})$
satisfying 
$$\hk_{\olx/X'}-J_{\olx/X'}-(ac-c-d+1)\Psi^*E\leq \hk_{\olx/X}-J_{\olx/X}$$
$$\leq \hk_{\olx/X'}-J_{\olx/X'}-(bc-c-d+1)\Psi^*E.
$$
In particular if $I_{X'}\o_{A'}(-aE)=I_X\o_{A'}$, then we have
$$\hk_{\olx/X'}-J_{\olx/X'}-(ac-c-d+1)\Psi^*E= \hk_{\olx/X}-J_{\olx/X}.$$

\end{cor}
\begin{proof} As $\dim X=d$, note that $K_{A'/A}=(c+d-1)E$ and apply Theorem \ref{comparison}.
\end{proof}

\begin{exmp}\label{2dim} Let $(X,x)$ be a  singularity on a reduced 2-dimensional 
scheme $X$ and let 
 $\varphi: X'\to X$ be the blow-up at  $x$. 
 If $(X,x)$ is MJ-canonical (MJ-log canonical) singularity,
 then $X'$ has MJ-canonical (MJ-log canonical) singularities.

Here, if $(X,x)$ is non singular, then $X'$ is also non singular and the 
above statement is trivial, therefore we may assume
that $(X,x)$ is a singular point.
For the both statements of the example, it is sufficient to prove 
$$\hk_{\olx/X}-J_{\olx/X}
\leq \hk_{\olx/X'}-J_{\olx/X'}$$
for a log resolution $\Psi:\olx\to X'$ of $\j_{X'}\j_X{\mathcal O}_{X'}$.
As $(X,x)$ is singular, we have $c\geq 1$ and
$$I_X\o_{A'}\subset I_{X'}\o_{A'}(-2E),$$
under the notation of Corollary \ref{blowup}.
Let $b=2$ and note that $bc-c-d+1=c -1\geq 0$.
Then apply the corollary, we obtain the required inequality
$$\hk_{\olx/X}-J_{\olx/X}
\leq \hk_{\olx/X'}-J_{\olx/X'}.$$

\end{exmp}

\begin{exmp}\label{3fold}
Let $(X,x)$ be a singular point in a 3-dimensional reduced scheme.
Assume $(X,x)$ is not a hypersurface double point.
Let $X'$ be the same as in Example \ref{2dim}.
If $(X,x)$ is MJ-canonical (MJ-log canonical ), then $X'$ has MJ-canonical (MJ-log 
canonical ) singularities.

As in Example \ref{2dim}, it is sufficient to prove that  $bc-c-d+1\geq 0$.
If $(X,x)$ is not a hypersurface singularity, then $c\geq 2 $ and we can take 
$b=2$ and obtain $bc-c-d+1=c-2\geq0$.
If $(X,x)$ is a hypersurface singularity of  multiplicity $\geq 3$, then we can take 
$b\geq 3$, therefore  $bc-c-d+1\geq 2-3+1=0$.
\end{exmp}

\begin{exmp} Let $S\subset \bP^{N-1}$ be a $(d-1)$-dimensional
non singular projectively normal closed subvariety defined by
polynomials of common degree $a$.
Let $X\subset \bA^N$ be its affine cone.
Then,
\begin{enumerate}
\item[(i)] $X$ is MJ-canonical if and only if $a\leq \frac{N-1}{N-d}$,
\item[(ii)] $X$ is MJ-log canonical if and only if $a\leq \frac{N}{N-d}$
\end{enumerate}
Let us check the MJ-log canonicity and MJ-canonicity of $X$.
Let $\Phi:A'\to \bA^{N}$ be the blow-up at the origin, $E$ the exceptional divisor and
 $X'$ the strict transform of $X$ in $A'$.
Then, by the defining equations of $X$ in $\bA^{N}$, we have
$$I_X\o_{X'}=I_{X'}\o_{X'}(-aE).$$
By Corollary \ref{blowup}, we have
$\hk_{\olx/X'}-J_{\olx/X'}-(ac-N+1)\Psi^*E= \hk_{\olx/X}-J_{\olx/X},$
with $ c=N-d$ for an appropriate log resolution $\Psi:\ola\to A'$.
Therefore we obtain  
\begin{equation}\label{segre}
\hk_{\olx/X'}-J_{\olx/X'}-(a(N-d)-N+1)\Psi^*E= \hk_{\olx/X}-J_{\olx/X}.
\end{equation}
Here, we note that $(X', E|_{X'})$ is non singular pair and $(X', \alpha E|_{X'})$ is log 
MJ-canonical if and only if $\alpha\leq 1$.
Then by the  equality (\ref{segre}) we have $X$ is MJ-log canonical if and only if $a(N-d)-N+1\leq 1$
which is equivalent to $a\leq \frac{N}{N-d}$.
On the other hand, if $a(N-d)-N+1\leq 0$ which implies $a\leq \frac{N-1}{N-d}$, we have that 
$X$ is MJ-canonical by the equality (\ref{segre}).
If $a(N-d)-N+1= 1$, then the equality  (\ref{segre}) implies $\ha(E; X,\ox)=0$, which yields
that $X$ is not MJ-canonical.
\end{exmp}

\begin{exmp}
Under the same setting as in the previous example, let $a=2$.
Then,
\begin{enumerate}
\item[(i)] $X$ is MJ-canonical if and only if $N\leq 2d-1$,
\item[(ii)] $X$ is MJ-log canonical if and only if $N\leq 2d$.
\end{enumerate} 
Note that these conditions on $N$ and $d$ are only  the necessary conditions for a general 
$X$ to be MJ-canonical and MJ-log canonical as are seen 
in Proposition \ref{emb}.

We can see that the cones of many homogeneous spaces enjoy these conditions.
For example, the cones of $G(2,5)\subset \bP^9$, $E_6\subset
\bP^{26}$ (\cite{z}) and 10-dimensional Spinor variety in $\bP^{15}$ (\cite{e}) are all
MJ-canonical.

Let $S_{rm}=\bP^r\times\bP^m\hookrightarrow \bP^{N-1}$ be the Segre embedding,
{\it i.e.,} the correspondence of the homogeneous coordinates is $(x_i)\times(y_j)
\mapsto (x_iy_j)$. Then the subscheme $S_{rm}$ is defined in $\bP^N$ by the 
equations $z_{ij}z_{kl}-z_{il}z_{kj}=0, (i=0,\ldots, r, j=0,\ldots, m)$, where $z_{ij}$'s are
homogeneous coordinates of $\bP^{N-1}$  $(N=(r+1)(m+1))$.
Let $X_{rm}\subset \bA^{N}$ be the affine cone over $S_{rm}$.
Then, as $d=r+m+1$, we have the following:
\begin{enumerate}
\item[(i)]
 $X_{rm}$ is MJ-log canonical if an only if $(r-1)(m-1)\leq 2$, 
\item[(ii)] $X_{rm}$ is MJ-canonical 
if and only if $(r-1)(m-1)\leq 1$.
\end{enumerate}
In particular, $X_{1m}$ and $X_{r1}$ are all MJ-canonical.
Here, we note that $X_{rm}$ is $\bQ$-Gorenstein if and only if
$r= m$.
Thus, if $r\neq 1$ or $m\neq 1$, then $X_{1m}$ and $X_{r1}$ are examples of
MJ-canonical singularities which are not $\bQ$-Gorenstein.

\end{exmp}

\begin{exmp}
Three dimensional terminal quotient singularities are determined as $\frac{1}{r}(s,-s,1)$
$(0<s<r, \gcd(s,r)=1)$ by \cite{ms}.
If $s\neq 1, r-1$, then the singularity $\frac{1}{r}(s,-s,1)$ is not MJ-log canonical 
singularities.
Indeed, the singularity is at the origin of $X=\spec k[x^r,y^r,z^r,xy,xz^{r-s},yz^s]=k[x_1,\ldots, x_6]/I$,
where $I=(x_3x_4-x_5x_6, x_1x_2-x_4^r, x_1x_3^{r-s}-x_5^r, x_2x_3^s-x_6^r)$.
Here, we note that the number of generators with order 2 is two.

Assume that $X$ has MJ-log canonical singularity at 0, then $\hmld(0, X, \ox)\geq 0$, therefore by the formula (\ref{hmld}) in Proposition \ref{description} we have $$\dim(\psi^X_{n0})^{-1}(0)\leq d(n+1)=3(n+1).$$
In particular for $n=2$, it follows $\dim(\psi^X_{2,0})^{-1}(0)\leq 9$.
Under the notation in \ref{field}, we have by Remark \ref{keisan}:
$$(\psi^X_{2,0})^{-1}(0)=\spec k[x_1^{(1)},\ldots, x_6^{(1)},x_1^{(2)},\ldots, x_6^{(2)}]/
(x_3^{(1)}x_4^{(1)}-x_5^{(1)}x_6^{(1)}, x_1^{(1)}x_2^{(1)})$$
whose dimension is greater than 9, a contradiction.
Therefore $X$ is not MJ-log canonical at 0.

\end{exmp}

\begin{rem}
  The MJ-discrepancy has good properties: Inversion of Adjunction on minimal log discrepancies,
  lower semi-continuity of MJ-minimal log discrepancies (\cite{dd}, \cite{Ishii}), ACC of MJ-log canonical
  thresholds (\cite{dd}).
  So, if every step in Minimal Model Program would preserves MJ-log canonicity, 
  we could prove MMP simply.
  But actually a divisorial contraction does not preserve MJ-log canonicity.
  Kawamata \cite{k} determined the divisorial contraction to a 3-dimensional terminal
  quotient singularity as a certain weighted blow-up. 
  By this we can prove that  every 3-dimensional terminal quotient singularity can be resolved 
  by the successive weighted blow-ups which are divisorial contractions.
  This gives a counter example to the expectation that MJ-log canonicity would be preserved under divisorial contractions. 
\end{rem}

\begin{prop}\label{slc}
  Assume $X$ is $S_2$ and $\bQ$-Gorenstein.
  If $(X, \a^t)$ is MJ-log canonical, then it is semi log canonical.
\end{prop}

\begin{proof}
The definition of a semi log canonical singularity requires $S_2$ and $\bQ$-Gorenstein property.
The additional conditions for a semi log canonical singularity are (\cite{ko}): 
\begin{enumerate}
  \item[(i)] $X$ is non singular or has normal crossing double singularities in codimension one.
  \item[(ii)] Let $\nu: X_\nu\to X$ be the normalization, $\a_\nu$ the pull back of $\a$ on $X_\nu$ 
  and $D_\nu$ the divisor on $X_\nu$ defined by the conductor $(\ox:\nu_*(\o_{X_\nu}))$.
  Then, $(X_\nu, \a_\nu^t\o_{X_\nu}(-D_\nu))$ is log canonical in the usual sense.
\end{enumerate}
Let $W$ be an irreducible component of singular locus of $X$ of codimension 1.
Then $\hmld(W; X,\ox)\geq 0$ implies $(\psi_{m0}^X)^{-1}(W)\leq d(m+1)$ by (\ref{hmld})
in Proposition \ref{description}.
As $\dim W=d-1$, for a general point $x\in W$ we have
$(\psi_{m0}^X)^{-1}(x)\leq dm+1$, then again by (\ref{hmld}) in Proposition \ref{description}, it follows 
$$\hmld(x; X,\ox)\geq d-1.$$
In this case, $\hmld(x; X,\ox)= d-1$ holds by \cite[Corollary 3.15]{Ishii}, \cite[Corollary 4.15]{dd}  and such $(X,x)$ is classified in \cite{ir} as to be normal crossing double or a pinch point when it is non normal.
As the pinch point locus is of codimension 2, we have the assertion (i).
The condition (ii) is equivalent to that the usual log discrepancy 
$a(E; X_\nu, \a_\nu^t\o_{X_\nu}(-D_\nu))\geq 0$
for every prime divisor $E$ over $X_\nu$.
As $\nu^*K_X\sim_{\bQ}K_{X_\nu}+D_\nu$, it is equivalent to $a(E, X, \a^t)\geq 0$.
By the same argument in the proof of Proposition \ref{mj-and-usual}, we obtain
$\ha(E, X, \a^t)\leq a(E, X, \a^t)$, which yields the assertion (ii).
Here, we note that in the proof of Proposition \ref{mj-and-usual} used \cite[Corollary 9.4]{e-Mus2}
which was stated under the condition that $X$ is normal.
But the proof of the corollary works also for non normal case.
\end{proof}

\begin{cor} Let $X$ be locally a complete intersection.
Then, $(X, \a^t)$ is MJ-log canonical if and only if it is semi log canonical.
\end{cor}

\begin{proof}
As  $X$ is locally a complete intersection, it is $S_2$.
Then, by Proposition \ref{slc}, 
if $(X, \a^t)$ is MJ-log canonical, it is semi log canonical.
Conversely, if $(X,\a^t)$ is semi log canonical, then by the condition (ii) of semi log canonical in the proof of Proposition \ref{slc}, we obtain 
$$\ha(E, X, \a^t)= a(E, X, \a^t)\geq 0$$ 
for every prime divisor $E$ over $X$ 
in the same way as in the proof above.
This yields the required equivalence.

\end{proof}

Here we note that the $S_2$ condition is necessary for a MJ-log canonical singularity 
to be semi log canonical.
Actually there is an example of MJ-log canonical singularity which does not satisfy
$S_2$ condition (see Example \ref{nonS2}).

De Fernex and Hacon introduced in \cite{dfh} the notions log canonical, log terminal singularities on an arbitrary normal variety. 
These are direct generalizations of usual log canonical, klt singularities for
$\bQ$-Gorenstein case. 
Actually they defined that $(X, \a^t)$ is log terminal (resp. log canonical) if 
there is $m\in \bN$ such that 
$$a_m(F; X,\a^t):=\ord_F(K_{m,Y/X})-t\val_F(\a)+1>0
\ \ (resp. \geq 0)$$
 for every prime divisor $F$ over $X$.
Here, in a local situation, as we can take an effective divisor $mK_X$, we can think a divisorial sheaf $\ox(-mK_X)$ as an ideal sheaf. 
Let  $Y\to X$ be a log resolution of an ideal $\ox(-mK_X)$ and define the effective divisor $D_m$
on $Y$ by $\ox(-mK_X)\oy=
\oy(-D_m)$.
Note that an arbitrary prime divisor $F$ over $X$ can appear on such a resolution $Y$. 
Under this notation we define the divisor 
$$K_{m, Y/X}=K_Y-\frac{1}{m}D_m$$
with the support on the exceptional divisor.
The following is the relation of this divisor and our MJ-discrepancy divisor.

\begin{lem}\label{lemdfh} Let $X$ be an affine normal variety and $m$ a positive integer.
Then, there is a log resolution $Y\to X$ of $\jx\ox(-mK_X)$ such that
$$\hk_{Y/X}-J_{Y/X}\leq K_{m, Y/X}.$$
\end{lem} 

\begin{proof}
Fix a log resolution $\varphi: Y \to X$ of $\j_X \ox(-mK_X)$.
Take a reduced complete intersection scheme $M\subset \bA^N$ of codimension $c$ such that $M$ 
contains $X$ as an irreducible component.
Then we have a sequence of homomorphisms of $\ox$-modules:
\begin{equation}\label{einmus}
(\wedge^d \Omega_X)^{\otimes m}\stackrel{\eta}\longrightarrow \omega_X^{[m]}
\stackrel{u}\longrightarrow (\omega_M|_X)^m.
\end{equation}
 By \cite[Proposition 9.1]{e-Mus2} the image of $u\circ \eta$ is written as 
 \begin{equation}\label{einmus1}
 (\j_M|_X)^m (\omega_M|_X)^m.
 \end{equation}
 Then take a pull-back of the sequence (\ref{einmus}):
\begin{equation}\label{einmus2}
\varphi^*(\wedge^d \Omega_X)^{\otimes m}\stackrel{\eta}\longrightarrow \varphi^*\omega_X^{[m]}
\stackrel{u}\longrightarrow \varphi^*(\omega_M|_X)^m.
\end{equation}
Define a divisor $D_m$ on $Y$ as $\oy(-D_m)=\ox(-mK_X)\oy$.

Then, we claim that  
\begin{equation}\label{=}
\varphi_*(\oy(D_m))=\ox(mK_X).
\end{equation}  
This is proved as follows:
As outside of the singular locus the both sheaves coincide and the right 
hand side
is reflecive, the inclusion  $\subset$ holds.
For the opposite inclusion, regard $\ox(mK_X)$ as $\ox(-mK_X)^*=\Hom_\ox(\ox(-mK_X), \ox)$.
For the claim, it is sufficient to prove that every homomorphism \\
$f: \ox(-mK_X) \to \ox$ 
comes from
a homomorphism  $\ox(-mK_X) \oy \to \oy$.
The homomorphism $f$  is lifted to $ f': \varphi^*( \ox(-mK_X)) \to \oy$.
Here, the torsion elements are mapped to zero by $f'$.
Therefore $f'$ factors through $\varphi^*( \ox(-mK_X))/Tor \to \oy$, where $Tor$ is the 
subsheaf consisting of the torsion elements of $\varphi^*( \ox(-mK_X))$.
But we can prove that $\varphi^*(\ox(-mK_X))/Tor =\oy(-D_m)$.
This completes the proof of the claim (\ref{=}).

By (\ref{=}), the 
sequence (\ref{einmus2}) factors as:
\begin{equation}\label{einmus3}
\varphi^*(\wedge^d \Omega_X)^{\otimes m}\stackrel{\eta'}\longrightarrow 
\oy(D_m)
\stackrel{u'}\longrightarrow  \varphi^*(\omega_M|_X)^m,
\end{equation}
where $u'$ is the dual map of the following:
$$\ox(-mK_X)\oy\leftarrow \ox(-mK_M|_X)\oy=( \varphi^*(\omega_M|_X)^m)^*.$$
As the second and the third sheaves in the sequence (\ref{einmus3}) are invertible,
we can write 
\begin{equation}\label{einmus7}
\im \eta'=I \oy(D_m)
\end{equation}
$$
\im u'=J_M  \varphi^*(\omega_M|_X)^m,
$$
with the ideals $I, J_M \subset \oy$.
Then, by (\ref{einmus1}), we obtain
\begin{equation}\label{einmus4}
I J_M=(\j_M|_X)^m\oy
\end{equation}
Consider all $M$ and define $J=\sum_M J_M$, then we have
$I J=(\sum_M\j_M^m)\oy$.
Then, by taking the integral closure of the both hand sides, we have
\begin{equation}\label{einmus5}
\overline{I J}=\overline{\j_X^m\oy}
\end{equation}

Now, given a prime divisor $F$ over $X$, it appears on a log resolution $\nu:Y'\to Y$ of 
$I J$.
Let $\psi: Y'\to X$ be the composite  $\nu\circ\varphi$.
Define effective divisors $B, C$ on $Y'$ such that $\oyy(-B)=I\oyy$ and $\oyy(-C)=J
\oyy$, then
\begin{equation}\label{einmus6}
B+C=mJ_{Y'/X}
\end{equation}
As $\psi$ factors through the Nash blow-up, the torsion free sheaf 
\newline 
$((\psi^*\wedge^d\Omega_X)/Tor)^{\otimes m}$
is invertible, therefore it is written as $\oyy(G)$ by a divisor $G$ on $Y'$.
Then, by the definition of $\hk_{Y'/X}$, we have $m\hk_{Y'/X}=mK_{Y'}-G$.
On the other hand, by (\ref{einmus7}) we have  $G=\nu^*D_m-B$      and by (\ref{einmus5}) we have
$$m\hk_{Y'/X}-mJ_{Y'/X}=mK_{Y'}-G-(B+C)=mK_{Y'}-\nu^*D_m-C\leq mK_{Y'}-\nu^*D_m
,$$
which completes the proof of the lemma.
\end{proof}

The following shows our MJ-klt and MJ-log canonical singularities become log terminal
and log canonical singularities in the sense of De Fernex and Hacon.

\begin{thm}\label{dFH} Assume that $X$ is normal. If a pair $(X, \a^t)$ is MJ-klt (resp. MJ-log canonical), then it is log terminal
(resp. log canonical) in the sense of De Fernex and Hacon.
\end{thm}

\begin{proof} 
Since the problem is local, we may assume that $X$ is a closed subvariety of 
the affine space $\bA^N$ of codimension $c$.
It is sufficient to prove for a fixed $m\in \bN$
$$\ha(F; X,\ox)\leq a_m(F;X,\ox)$$
for every prime divisor $F$ over $X$.
As noted above, we may assume that $\ox(-mK_X)$ is an ideal sheaf of $\ox$.
By the lemma we have a log resolution $\varphi:Y\to X$ of $\jx\ox(-mK_X)$ 
such that the inequality 
$$\hk_{Y/X}-J_{Y/X}\leq K_{m, Y/X}$$
holds.
Then, note that  every  resolution $\psi:Y'\to X$ factoring through $\varphi$ satisfies the
inequality. 
Therefore, every prime divisor $F$ over $X$ appears on a resolution on which the
inequality holds,
which yields $\ha(F;X,\ox)\leq \a_m(F; X,\ox)$.
\end{proof}

By \cite[Theorem 1.2]{dfh} a pair $(X, \a^t)$ is log terminal (resp. log canonical) in De Fernex
and Hacon's sense if and only if there is a boundary $\Delta$ (it means that $\Delta$ is
a $\bQ$-divisor such that $[\Delta]=0$ and $K_X+\Delta$ is a $\bQ$-Cartier divisor)
such that $((X,\Delta), \a^t)$ is klt (resp. log canonical) in the usual sense.
Therefore we obtain the following corollary.

\begin{cor}\label{usual} If a pair $(X, \a^t)$ is MJ-klt (resp. MJ-log canonical), 
then there is a boundary $\Delta$ on $X$ such that $((X, \Delta), \a^t)$ is klt 
(resp. log canonical) in the usual sense.
\end{cor}

In \cite{dfh}, De Fernex and Hacon also introduced a multiplier ideal for a pair 
$(X, \a^t)$ with a normal variety $X$ and an ideal $\a$ on $X$.
First for $m\in \bN$ they defined $m$-th ``multiplier ideal" as follows:
$$\j_m(X, \a^t)=\varphi_*(\oy(\ulcorner K_{m,Y/X}-tZ\urcorner)),$$
where $\varphi:Y\to X$ is a log resolution of $\a \ox(-mK_X)$
and let $\a\oy=\oy(-Z)$. They proved that the family of  ideals $\{\j_m(X,\a^t)\}_m$ has
the unique maximal element and call it the multiplier ideal of $(X, \a^t)$ and denote it
by $\j(X, \a^t)$.
The following is the relation between their multiplier ideal  and our MJ-multiplier ideal,
which follows immediately from Lemma \ref{lemdfh}

\begin{thm}\label{mjmulti} Let $(X,\a^t)$ be a pair with a normal variety $X$, an ideal $\a$ on $X$ and
$t\in \bR_{\geq 0}$. Then the following inclusion holds for every $m\in \bN$:
$$\hj(X,\a^t)\subset \j_m(X,\a^t),$$
in particular 
$$\hj(X,\a^t)\subset \j(X,\a^t).$$
\end{thm}

The following proposition is an application of Inversion of Adjunction, where
the first result is contained in Corollary \ref{usual}, but we think that it 
makes sense to give a direct proof  without using the result
of \cite{dfh}. 

\begin{prop} 
\begin{enumerate}

\item[(i)]
Let $X$ be an MJ-canonical variety, then there exists an effective $\bQ$-divisor $\Delta$ on 
$X$ such that $(X, \Delta )$ is klt in the usual sense.
\item[(ii)]
Let $X$ be MJ-log canonical and $W$ be a minimal MJ-log canonical center,
then there exists an effective $\bQ$-divisor $\Delta$ on 
$W$ such that $(W, \Delta )$ is klt in the usual sense.
\end{enumerate}
\end{prop}

\begin{proof} As $X$ is MJ-canonical,
 it is irreducible and normal by \cite{dd} or \cite{eim}.
If there exist an open covering $\{ U_i\}$ of $X$ (resp. $W$) and an effective $\bQ$-divisor $\Delta_i$
on $U_i$  such that $(U_i, \Delta_i)$ is klt for each $i$, then by \cite[Theorem 1.2]{dfh} there exists 
a global $\bQ$-divisor $\Delta$ on $X$ (resp. $W$) such that $(X, \Delta )$ (resp. $(W,\Delta)$) is klt.
So we may assume that $X$ is affine for both statement (i) and (ii).
 Let $X$ be embeded in a non singular affine variety $A$ with codimension $c$ and the 
defining ideal $I_X$. 

(i)  As $X $ is MJ-canonical, we have $\hmld(Z; X,\ox)\geq 1$ for every proper closed subset
$Z\subset X$.
By Inversion of Adjunction we have 
$$\mld(Z;A,I_X^c)\geq 1.$$
On the other hand, for any point $\eta\not\in X$ in $A$
$$\mld(\eta;A,I_X^c)=\mld(\eta;A, \oa)\geq 1$$
because $A$ is non singular.
Finally for the generic point $\eta$ of $X$, we have
$$\mld(\eta,A,I_X^c)=0.$$
Hence, $X$ is the unique log canonical center of $(A, I_X^c)$.

Now, take a log resolution $\varphi:\ola\to A$ of $(A, I_X)$.
Take a general element 
$g\in I_X^{2c}$ and
let $D_0$ be the zero locus of $g$ on $A$ and then let $D=\frac{1}{2}D_0$.
Then, by the generality of $g$, the morphism $\varphi$ is also a log resolution of
$(A, D_0)$ and for every exceptional prime divisor $E_i$ on $\ola$ we have
$$a(E_i; A, I_X^c)=a(E_i;A,D).$$
As $(A, D)$ is klt outside of $X$, $(A,D)$ has also unique log canonical center $X$.
Then, by Local Subadjunction Formula by Fujino and Gongyo (\cite{fg}),
there exists a $\bQ$-divisor $\Delta$ on $X$ such that 
$(X, \Delta)$ is klt.

(ii) By Inversion of Adjunction we have 
$\hmld(W; A, I_X^c)=0$ and \\
$ \hmld(Z;A, I_X^c)\geq 0$ for every strictly  proper closed subset 
$Z$ of $X$.
By the minimality of $W$ we have
$$\hmld(Z; A, I_X^c)>0 $$
for every  strictly proper closed subset $ Z\subset W.$
We also have 
$\hmld(\eta; A, I_X^c)=0  $ for the generic point $\eta$ of an irreducible component of $X$
and $\hmld(\eta;A, I_X^c)\geq 1$ for any point $\eta\not\in X$ in $A$.
Therefore, $(A, I_X^c)$ is log canonical and $W$ is a minimal log canonical center of $(A, I_X^c)$.
Then, by the same argument as in (i), we have $(W,\Delta)$ is klt for some 
boundary $\Delta$.

\end{proof}

\section{Deformations}

In this section we prove that MJ-canonical singularities and MJ-log canonical singularities 
are preserved under small deformations.
First we start with the strengthening of Inversion of Adjunction. 
Proposition \ref{inversion} does not hold for singular $A$ in general (see, \cite[Example 3.13 ]{Ishii}), but
if $X$ is a complete intersection in a singular $A$, then it holds.

\begin{cor}[Strong Inversion of Adjunction] \label{stin} Let $A$ be an affine connected reduced equidimensional scheme of finite type over $k$ of dimension $d+c$ containing $X$ as a complete
intersection, {\it i.e.,} $X$ is defined by $c$ equations $f_1=f_2=\cdots =f_c=0$ in
$A$. 
Then, the following hold:
\begin{enumerate}
\item[(i)]
Assume $X$ is reduced and let $W$ be a strictly  proper closed subset of $X$.
 Let $\widetilde\a\subset \o_A$ be an ideal such that its image ${\a}:=\widetilde\a\o_X\subset \o_X$ is non-zero on each irreducible component of $X$.
 Then,
$$\hmld(W; X,{\a^t})=
\hmld(W;A,\widetilde\a^t (f_1,\ldots, f_c)^c).$$
\item[(ii)] If $A$ satisfies $S_2$, $c=1$ and $(A, (f_1))$ is MJ-log canonical,
then automatically $X$ is reduced and the formula in (i) holds.

\end{enumerate}
\end{cor}

\begin{proof} We may assume that $A$ is embedded into the affine space $\bA^N$.
By using the same idea as in Remark \ref{linear}, we can construct an embedding
$A\subset \bA^{N+c}$ such that there exists a linear subspace $L$ of codimension $c$ in 
$\bA^{N+c}$ satisfying $L\cap A=X$. Denote $B=\bA^{N+c}$. 
Let $\overline{\a}\subset \o_{B}$ be an ideal such that $\widetilde\a=\overline{\a}\o_A$ and
let $\a' =\overline{\a} \o_{L}$.
Then  we have $\a=\a'\ox$.
Let $I_{X/L}$, $I_{A/B}$, $I_{L/B}$ be the defining ideals of $X$ in $L$, $A$ in $B$, $L$ in $B$,
respectively. Then $L\cap A=X$ implies that $I_{X/A}=I_{L/{B}}\o_A$ and $I_{X/L}=I_{A/B}\o_L$.

Noting that  $B$ and $L$ are non singular, apply Proposition \ref{inversion} for
$X\subset L$, $L\subset B$ and $A\subset B$. Then we obtain
$$ \hmld(W; X,{\a^t})=
\mld(W;L,{\a'}^t  (I_{X/L})^{N-d}),$$
$$ \mld(W;L, {\a'}^t (I_{X/L})^{N-d})
=\mld(W; B,{\overline\a}^t(I_{A/{B}})^{N-d}(I_{L/{B}})^c),$$
$$\hmld(W, A,\widetilde\a^t (I_{X/A})^c)=\mld(W; B,{\overline\a}^t(I_{L/{B}})^c(I_{A/{B}})^{N-d}).$$
The required equality in (i) follows from just composing these equalities.

For the proof of (ii), first we see that $A$ is smooth at the generic point of every irreducible
component of $X$.
This is proved as follows: Assume $A$ is not smooth at the generic point $\eta$ of an irreducible component of $X$.
Then as $\hmld (\eta; A, (f_1))\geq 0$, we have $\hmld (\eta;A,\oa)\geq 1$, 
which implies that $A$ is MJ-canonical around general points of $\overline{\{\eta\}}$.
But MJ-canonical singularities are normal, a contradiction.
By restricting $A$ to a neighborhood of $X$ we may assume that  $Z=\sing A$ is of 
codimension $\geq 2$.
Let $A_0=A\setminus Z$, then $(A_0, (f_1))$ is MJ-log canonical, but it is equivalent
to that $(A_0, (f_1))$ is log canonical in the usual sense, because $A_0$ is non singular.
Therefore $f_1$ is reduced on $A_0$. 
For every open subset $U\subset A$, define $U_0:=U\cap A_0$
Consider the exact sequence:
$$0\to \Gamma(U, (f_1))\to \Gamma(U_0, (f_1))\to H_{Z\cap U}^1(U, (f_1)).$$
Since the ideal sheaf $(f_1)$ is principal, the last term $H_{Z\cap U}^1(U, (f_1))$
is isomorphic to $H_{Z\cap U}^1(U, \o_U)$
and this  is  0,  because 
 $A$ is $S_2$.
 Therefore by the exact sequence, we obtain $(f_1)=i_*(f_1|_{A_0})=i_*(\sqrt{(f_1|_{A_0})})\supset \sqrt{(f_1)}$, where $i:A_0\hookrightarrow A$ is the inclusion. 
 This shows that the ideal $(f_1)$ is reduced on $A$.
 Once we know that $X$ is reduced we can apply (i) to obtain the formula of 
 $\hmld$.

\end{proof}

\begin{defn}
Let $T$ be a reduced scheme of finite type over $k$ and $0\in T$ a closed point. 
Let $\pi: X\to T$ be a surjective morphism   
with equidimensional reduced fibers $X_\tau=\pi^{-1}(\tau)$ of common dimension $r$ for all closed points $\tau\in T$.
 Then $\pi:X\to T$ is called a deformation of $X_0$ with the parameter space $T$. 
 
 If moreover a pair $(X, \a^t)$ is given, $\a^t_\tau=\a^t\o_{X_\tau}$ are not zero 
 on each irreducible component of $X_\tau$ for
 all $\tau\in T$ and $\pi:X\to T$ is a deformation of $X_0$,
 then the family $\{(X_\tau, \a^t_\tau)\}_{\tau\in T}$ is called a deformation of $(X_0, \a^t_0)$.
 
 From now on, for a morphism $\pi: Z\to T$ from some scheme $Z$ to the parameter space 
 $T$, we denote the fiber $\pi^{-1}(\tau)$ by $Z_\tau$.
\end{defn}

\begin{lem}\label{generalfiber}
  Let $\pi:X\to T$ be a deformation of $(X_0, \a_0^t)$ $(0\in T)$ given by a non zero ideal
  $\a\subset\ox$.
  Then, there exists an open dense subset $T_0\subset T$ and a log resolution 
  $\varphi: Y\to X$ of $(X, \a\jx)$ such that for every $\tau\in T_0$ the following hold:
  \begin{enumerate}
 \item[(i)] $\varphi_\tau: Y_\tau\to X_\tau$ is a log resolution of $(X_\tau, \a_\tau\j_{X_\tau})$;
 \item[(ii)] $(\hk_{Y/X}-J_{Y/X}-tZ)|_{Y_\tau}=\hk_{{Y_\tau}/X_{\tau}}-J_ {Y_\tau}/X_{\tau}-tZ_\tau,$
  \end{enumerate}
where $\varphi_\tau $ is the restriction of $\varphi$ onto the fiber $Y_\tau=(\pi\circ\varphi)^{-1}(\tau)$ and $\a\oy=\oy(-Z)$, $\a_\tau\o_{Y_\tau}=\o_{Y_\tau}(-Z_\tau)$.

In particular, $(X|_{T_0},\a^t)$ is MJ-log canonical (resp. MJ-canonical ) if and only if
$(X_\tau, \a^t_\tau)$ is MJ-log canonical (resp. MJ-canonical ) for every $\tau \in T_0$.
\end{lem}

\begin{proof} As it is sufficient to prove the existence of such an open subset of $T$ on each irreducible 
component, we may assume that $T$ is irreducible.
 Let $r$ be the common dimension of the fiber $X_\tau$ for closed points $\tau\in T$.
Let $\j_{X/T}$ be the $r$-th Fitting ideal of $\Omega_{X/T}$.
By Proposition \ref{eefact}, we can take a factorizing resolution $\Phi:\ola\to A$ of $X$ in $A$ 
with the strict transform $Y$ of $X$ in $\ola$ such that the restriction $\varphi:Y\to X$ of 
$\Phi$ is a log resolution of $(X,\a\jx\j_{X/T})$.
Let $E_i$ $(i=1,\ldots, s)$ be an exceptional prime divisor of $\Phi$.
Then, by the generic smoothness theorem, there is an open dense subset $T_0$ of $T$  such that $E_{i_1}\cap\cdots\cap E_{i_j}$, $E_{i_1}\cap\cdots\cap E_{i_j}
\cap Y$, $Y$,  $\ola$, $A$ are smooth over $T_0$ for all collections $\{i_1,\ldots, i_j\}$ if they are not empty.
On the other hand, since $\Phi$ is a factorizing resolution of $X$ in $A$, we have an effective 
divisor $R$ on $\ola$ such that $I_X\ooa=I_Y\ooa(-R)$. 
Replacing $T_0$ by a smaller open subset if necessary, we may assume that
the support of $R$ does not contain  $\ola_{\tau}$ $(\tau\in T_0)$.
By restricting this equality on the fiber of $\tau$, we have
$$I_X{\o_{\ola_\tau}}=I_{Y_\tau}{\o_{\ola_\tau}}(-R|_{A_\tau}).$$
Because of this, $\Phi_\tau:\ola_\tau\to A_\tau$ is a factorizing resolution of $X_\tau$ in $A_\tau$ for every $\tau \in T_0$.

Then, by $\Omega_{X/T}\otimes \o_{X_\tau}=\Omega_{X_\tau}$ and the functoriality  of Fitting ideals, we have  $\j_{X/T}\o_{X_\tau}=\j_{X_\tau}$ for every $\tau\in T_0$.
This shows that $\varphi_\tau$ is a log resolution of $(X_\tau, \a_\tau\j_{X_\tau})$.

By the Lemma \ref{lemcom} we have $$\hk_{Y/X}-J_{Y/X}= (K_{\ola/A}-cR)|_{Y},$$
where $c=\codim (X, A)$. Noting that $c$ is also the codimension of $X_\tau$ in $A_\tau$ for
a closed point $\tau\in T_0$, we have
$$\hk_{Y_\tau/X_\tau}-J_{Y_\tau/X_\tau}= (K_{\ola_\tau/A_\tau}-cR|_{\ola_\tau})|_{Y_\tau}.$$
Since $(K_{\ola/A})|_{\ola_\tau}=K_{\ola_\tau/A_\tau}$, we obtain for $\tau\in T_0$
$$(\hk_{Y/X}-J_{Y/X})|_{Y_\tau}=\hk_{{Y_\tau}/X_{\tau}}-J_ {{Y_\tau}/X_{\tau}}.$$
For the  statement (ii) we have only to note that  $Z|_{Y_\tau}=Z_\tau$ for $\tau\in T_0$.

\end{proof}

\begin{thm}\label{deform-log}
  Let $\{(X_\tau, \a^t_\tau)\}_{\tau\in T}$ be a deformation of $(X_0,\a^t_0)$.
  Assume $(X_0,\a^t_0)$ is MJ-log canonical at $x\in X_0$.
  Then there are neighborhoods $X^*\subset X$ of $x$ and $T^*\subset T$ of $0$ such that
  $(X^*_\tau, \a^t_\tau|_{X^*_\tau})$ is MJ-log canonical for every closed point $\tau\in T^*$.
\end{thm}

\begin{proof}
  The statement is reduced to the case that $T$ is a non singular curve.
  Then $X_0$ is defined by one equation, say $f=0$, and $\dim X_0$ is one less than $\dim X=d$.
  By applying Corollary \ref{stin}, we have
$$\hmld(x; X_0,{\a_0^t})=
\hmld(x;X,\a^t  (f)).$$
By the assumption we have $\hmld(x; X_0,{\a_0^t})\geq 0$ which implies \\
$\hmld(x;X,\a^t (f))\geq 0$ and 
therefore $\hmld(x;X,\a^t )\geq 0$. 
Then, by Proposition \ref{atx} there is an open neighborhood $X^*\subset X$ of $x$ such that $(X^*, \a^t|_{X^*})$ 
is MJ-log canonical.
Then, by the last statement of Lemma \ref{generalfiber}, there exists an open subset 
$T*$ such that 
$(X^*_\tau, \a^t_\tau|_{X^*_\tau})$ is MJ-log canonical for every $\tau\in T^*$.
\end{proof}

\begin{rem}
  Replacing $X$ by  a small neighborhood of $x$, we can assume that 
  $X\subset T\times \bA^N$, since the morphism $X\to T$ is of finite type.
  If $T$ is non singular, then $A=T\times\bA^N\to T$ is a smooth morphism of non singular 
  varieties.
  For $(X, \a^t)$, take $\widetilde\a\subset A$ as the pull back of $\a$ by the canonical surjective
  map $\o_A\to \ox$.
  Then, we can prove that $(X_\tau, \a_\tau^t)$ is MJ-log canonical if and only if 
  $(A_\tau,\widetilde\a (I_{X_\tau})^c) $ is log canonical.
  By using this fact, Theorem \ref{deform-log} can also be  proved by discussions only on $A$ and $A_\tau$.
\end{rem}

For the similar statement as Theorem \ref{deform-log} for MJ-canonical singularities we need some notions and a lemma.

\begin{defn}
Let $A$ be a non singular variety and $\eta\in A$ a (not necessarily closed) point.
For a cylinder $C\subset \mathcal{L}^\infty (A)$ we define the codimension of $C\cap \psi_{\infty 0}^{-1}(\eta)$ as follows:
$$\codim C\cap \psi_{\infty 0}^{-1}(\eta):=\codim (
\overline{\psi_{\infty m}(C\cap \psi_{\infty 0}^{-1}(\eta))}, \lm(A)),$$
for $m\gg 0$, where $\psi_{\infty m}:\mathcal{L}^\infty(A)\to \lm(A)$ is the canonical projection.

Here, note that the value of the right hand side is constant for $m\gg 0$, where $C=
\psi_{\infty n}^{-1}(S)$ for $S\subset \ln(A)$.
\end{defn}


\begin{lem}\label{generic}
Let $A$ be a non singular variety, $\eta\in A$  a (not necessarily closed) point and
 $\a\subset\oa$ $(i=1,\ldots, r)$ a  nonzero ideal.
 Then
 $$\mld(\eta; A, \a^t)= \inf\left\{\codim\left(\cont^{m}(\a)\cap\psi_{\infty,0}^{-1}(\eta)\right)-mt\right\}.$$
\end{lem}

\begin{proof} First we prove the inequality $\geq$.
Let $E$ be a prime divisor over $A$ with the center $\overline{\{\eta\}}$ and let $v=\val_E$.
Let $m=v(\a)$, then there exists a open dense subset $C\subset C_A(v)$
such that $C\subset \cont^m(\a)\cap\psi_{\infty,0}^{-1}(\eta),$ where
$C_A(v)$ is the maximal divisorial set (for definition see, for example, \cite{Ishii}) in $\li(X)$ corresponding to $v$.
This is because the generic point $\alpha\in C_A(v)$ has $\ord_\alpha(\a)=m$ by \cite{DEI}   and
the center of $\alpha$ is $\eta$.
Then $$\ord_E(K_{A'/A})-tv(\a)+1=\codim(C_A(v))-mt$$
$$\geq \codim(\cont^m(\a)\cap
\psi_{\infty,0}^{-1}(\eta))-mt,$$
where $Y\to X$ is a log resolution of $\a$ such that $E$ appears on $Y$.
Here, note that we use the equality $\ord_E(K_{A'/A})+1=\codim(C_A(v))$ proved in \cite{DEI}.
This completes the proof of $\geq$. 

Next we prove the opposite inequality $\leq$. 
We may assume that  
$$\ord_E(K_{A'/A})-t\val_E(\a)+1\geq 0$$ 
for every prime divisor $E$ over $X$ with the center $\overline{\{\eta\}}$,
because otherwise the claimed inequality is trivial.
For an arbitrary $m\in \bN$ take $\zeta\in \cont^m(\a)\cap\psi_{\infty,0}^{-1}(\eta)$ such that 
$\overline{\{\zeta\}}$ is an irreducible component of 
$\overline{\cont^m(\a)\cap{\psi_{\infty,0}^{-1}(\eta)}}$ and
$$\overline{\psi_{\infty, s}(\zeta)}\subset \overline{\psi_{\infty s}({\cont^m(\a)\cap\psi_{\infty,0}^{-1}(\eta)})},
\ \ \ s\geq m$$
gives the codimension of $\cont^m(\a)\cap\psi_{\infty,0}^{-1}(\eta)$.
Then, we have 
$$\overline{\{\zeta\}}=\psi_{\infty, s}^{-1}(\overline{\psi_{\infty, s}(\zeta)}),$$
which is an irreducible cylinder.
Then, a divisorial valuation $v=q \val_E$ over $A$ corresponds to this cylinder
(\cite[Propositions 2.12, 3.10]{DEI}).
Here, we note that $E$ is a prime divisor with the center $\overline{\{\eta\}}$
and $m=q\val_E(\a)$.
By the maximality of $C_A(v)$, we have 
$$\overline{\{\zeta\}}\subset C_A(v).$$
Hence, we have
$$\codim (\cont^m(\a)\cap \psi_{\infty, 0}^{-1}(\eta))-tm \geq \codim C_A(v)-tm$$
$$=q(\ord_E(K_{A'/A})+1)-q\val_E(\a)\geq \ord_E(K_{A'/A})-t\val_E(\a)+1,$$
which gives the inequality $\leq$ in the lemma as required.

\end{proof}

\begin{rem}\label{generic-remark} Let $A$ and $\eta$ be as above.
Let  $\a_i\subset\oa$ $(i=1,\ldots, r)$  be nonzero ideals and $t_i$ $(i=1,\ldots, r)$ non-negative
 real numbers, then the following holds:
\newline
 \vskip.2truecm
 $\mld(\eta; A, \a_1^{t_1}\cdots\a_r^{t_r})$
 $$=\inf\left\{\codim\left(\cont^{m_1}(\a_1)\cap
 \cdots\cap\cont^{m_r}(\a_r)\cap\psi_{\infty,0}^{-1}(\eta)\right)-\sum_i m_it_i\right\}$$
$$=\inf\left\{\codim\left(\cont^{\geq m_1}(\a_1)\cap
 \cdots\cap\cont^{\geq m_r}(\a_r)\cap\psi_{\infty,0}^{-1}(\eta)\right)-\sum_i m_it_i\right\}.$$
Here, the first equality is  proved in the similar way as in Lemma \ref{generic} and
the second equality follows from the same argument as the proof of \cite[Proposition 3.7]{Ishii}.

\end{rem}

\begin{thm}\label{deform-cano}
  Let $\{(X_\tau, \a^t_\tau)\}_{\tau\in T}$ be a deformation of $(X_0,\a^t_0)$.
  Assume $(X_0,\a^t_0)$ is  MJ-canonical at $x\in X_0$.
  Then there are neighborhoods $X^*\subset X$ of $x$ and $T^*\subset T$ of $0$ such that
  $(X^*_\tau, \a^t_\tau|_{X_\tau^*})$ is MJ-canonical for every $\tau\in T^*$.
\end{thm}

\begin{proof}
 As in Theorem \ref{deform-log}, we reduce to the case that $T$ is a non singular curve.
 If the statement does not hold, then there is a horizontal irreducible closed subset $W$ 
 ({\it i.e.,} $W$ dominates $T$ ) such that $x\in W$ and $\hmld(W; X,\a^t)<1$.
 Replacing $X$ by  a small neighborhood of $x$ we can assume that 
  $X\subset T\times \bA^N=A$. 
 Then, by Inversion of Adjunction, we have $\mld(W; A, \widetilde\a^tI_X)<1$,
 where $\widetilde\a\subset \oa$ is an ideal such that $\a=\widetilde\a\ox$.
 Then,
 $$\mld(\eta;A, \widetilde\a^tI_X)<1.$$
 Therefore, there exists a  prime divisor $E$ over $A$
 with the center  $W$ and $a(E; A, \widetilde\a^tI_X)<1$.
 Then, by Lemma \ref{generalfiber}, there is an open dense subset  $T_0\subset T$
 such that 
  \begin{equation}\label{mujun}
  \mld(\eta_\tau^{(i)}; A_\tau,\widetilde\a_\tau^tI_{X_\tau})<1\ \ \mbox{for}\ \   \tau \in T_0
  \end{equation}
where $\eta_\tau^{(i)}$ is the generic point of an irreducible component $W_\tau^{(i)}$ of $W_\tau$.
 \begin{equation}\label{description1}\mld(W_\tau;A_\tau,\widetilde\a_\tau^tI_{X_\tau})=\ \ \ \ \ \ \ \ \ \ \ \ \ \ \ \ \ \ \ \ \ \ \ \ \ \ \ \ \ \ \ \ \ \  \ \ 
 \end{equation}
 $$\hmld(W_\tau; X_\tau,{\widetilde\a_\tau^t})=\inf_{m,n}\{(M+1)N-(m+1)t-(n+1)c $$
$$\ \ \ \ \ \ \ \ \ \ \ \ \ \ \ \ \ \ \ \ - \dim\left(\psi_{Mm}^{-1}(\lm(Z_\tau))
\cap \psi_{Mn}^{-1}(\ln(X_\tau))\cap \psi_{M0}^{-1}(W_\tau)\right)\},$$
where $M=\max\{m, n\}$ and $\psi_{Mn}:\mathcal{L}^M(A)\to\ln(A)$ and so on .
Now, fix $m, n$.
For simplicity let us assume $M=n$. (for the other case $M=m$, the proof is similar).
Let $\ln(X/T)$ be the relative $n$-jet scheme with respect to $\pi:X\to T$.
It is defined as  
$$\ln(X/T):=\pi_n^{-1}(\Sigma^n(T) )\subset \ln(X),$$
where $\pi_n:\ln(X)\to \ln(T)$ is the morphism of $n$-jet schemes induced from $\pi:X\to T$
and $\Sigma^n(T)\subset \ln( T)$ is the locus of trivial $n$-jets on $T$.
Then, note that $\left(\ln(X/T)\right)_\tau=\ln(X_\tau)$.
Denote the canonical projection $\ln(X/T)\to \lm(X/T)$ by $\rho^X_{nm}$, then $\rho^X_{nm}|_{(\ln(X/T))_\tau)}$
is the canonical projection $\ln(X_\tau)\to \lm(X_\tau)$. 

Then the description in (\ref{description1}) is 
$$\mld(W_\tau;A_\tau,\widetilde\a_\tau^tI_{X_\tau})=\inf_{m,n}\{(M+1)N-(m+1)t-(n+1)c $$
$$\ \ \ \ \ \ \ \ \ \ \ \ \ \ \ \ \ \ \ \ - \dim\left(\left((\rho^X_{nm})^{-1}(\lm(Z_\tau))\right)
\cap \ln(X_\tau)\cap (\rho^X_{n0})^{-1}(W_\tau)\right)\}.$$
We  denote the inside of the bracket $\{\ \ \}$ of the right hand side by 
$$(M+1)N-(m+1)t-(n+1)c - R_{n,m,\tau}.$$ 
Let 
$$\mathcal R:=(\rho^X_{nm})^{-1}(\lm(Z/T))
\cap \ln(X/T)\cap (\rho^X_{n0})^{-1}(W)$$
and consider the restricted morphism $\rho:   \mathcal R\to W$
of $\rho^X_{n0}:\ln(X/T)\to X$.

Here, note that $R_{n,m,\tau}=\dim \rho^{-1}(W_\tau)$ for every $\tau\in T$.
Assume $\dim W=s$, then $\dim W_\tau=s-1$ since $T$ is a non singular curve and therefore
$W_\tau$ is a hypersurface in $W$.
Therefore 
$$R_{n,m,0}=\dim \rho^{-1}(W_0)\geq \dim \rho^{-1}(y)+s-1$$ 
for general closed point $y\in W$.
Take $\tau\in T$ such that  $y\in W_\tau^{(i)}\subset W_\tau$, then  
$$ \dim \rho^{-1}(y)+s-1=\dim \overline{\rho^{-1}(\eta_\tau^{(i)})}$$
Noting that $$\mld(\eta_\tau^{(i)}; A_\tau, \widetilde\a I_{X_\tau})=
\inf_{n,m}\{(M+1)N-(m+1)t-(n+1)c-\dim \overline{\rho^{-1}(\eta_\tau^{(i)})}\}$$
by Lemma \ref{generic}.
From (\ref{mujun}) we obtain
$$1\leq \mld(W_0;A_0,\widetilde\a_0I_{X_0})\leq \mld(\eta_\tau^{(i)}; A_\tau, \widetilde\a I_{X_\tau})
< 1,$$
which is a contradiction.
\end{proof}

As a corollary, we obtain a sufficient condition for a hypersurface singularity
not to be MJ-log canonical or MJ-canonical.
Terminologies ``non degenerate", ``Newton polygon" in the corollary can be referred in \cite{jalg}.

\begin{cor}\label{degenerate}
 Let $(X, 0)\subset (\bA^{d+1},0)$ be a reduced hypersurface singularity defined by an
 equation $f=0$.
Denote the Newton polygon of $f$ in $ \bR^{d+1}$ by $\Gamma(f)$. 
Then the following hold:
\begin{enumerate}
\item[(i)] If  
 ${\bf 1}=(1,\ldots, 1)\not\in \Gamma(f)$, then $(X,0)$ is not MJ-log canonical.
 \item[(ii)]
 If  
 ${\bf 1}=(1,\ldots, 1)\not\in \Gamma(f)^0$, then $(X,0)$ is not MJ-canonical.
 Here, $\Gamma(f)^0$ means the interior of $\Gamma(f)$.
\end{enumerate}
\end{cor}

\begin{proof}
  It is known that the statements hold for non-degenerate $f$ (see \cite[Corollary 1.7]{jalg}), since in this case MJ-canonical 
  (resp. MJ-log canonical) is equivalent to canonical (resp. log canonical ) in the usual sense.
  Let $f$ be possibly degenerate and  assume ${\bf 1}\not\in \Gamma(f)$.
  Perturb the coefficients of $f$ to obtain $f_{\epsilon}$ with $\Gamma(f_{\epsilon})=
  \Gamma(f)$. Let $\epsilon\in T:=\bA^r$ and $f=f_0$.
  Then $f_\epsilon$ $(\epsilon\in T)$ gives a deformation of hypersurfaces $X_\epsilon$.
  Then for general $\epsilon$, $f_\epsilon$ is non degenerate, therefore  ${\bf 1}\not\in \Gamma(f_\epsilon)$ implies that $X_\epsilon$ is not log canonical.
  Hence, $X_0=X$ is not MJ-log canonical by Corollary \ref{deform-log}.
    For the statement of MJ-canonical follows by using Theorem \ref{deform-cano} in the 
  similar way as above.
\end{proof}

\begin{prop}[Lower semi continuity of MJ-minimal log discrepancy]\label{lowsc}
 Let $\{(X_\tau, \a^t_\tau)\}_{\tau\in T}$ be a deformation of $(X_0,\a^t_0)$ and let 
 $\pi:X\to T$ be the morphism giving the deformation.
 Let $\sigma: T\to X$ be a section of $\pi$.
 Then, the map $T\to \bR, \tau \mapsto \hmld(\sigma(\tau), X_\tau, \a^t_\tau)$ is lower 
 semi continuous.
\end{prop}

\begin{proof}
For the statement of the proposition, we may assume that $T$ is irreducible.
We use the same notation as in the proof of Theorem \ref{deform-cano}.
First note that there is a non-empty open subset $T^*\subset T$ 
such that $\hmld(\sigma(\tau), X_\tau, \a^t_\tau)$ is constant for all $\tau\in T^*$.
This is proved as follows: Take a log resolution $\varphi: Y\to X$ of $(X, \a\j_X\j_{X/T}I_\Sigma)$,
where $I_\Sigma$ is the defining ideal of the section $\Sigma:=\im \sigma$. 
Then, by Lemma \ref{generalfiber}, there exists a non empty open  subset $T^*\subset T$
such that for every $\tau \in T^*$ the restriction $\varphi_\tau: Y_\tau \to X_\tau$ is 
a log resolution of $(X_\tau, \a_\tau \j_{X_\tau}{\frak m}_{X_\tau, \sigma(\tau)})$
and $$(\hk_{Y/X}-J_{Y/X}-tZ)|_{Y_\tau}=\hk_{{Y_\tau}/X_{\tau}}-J_ {{Y_\tau}/X_{\tau}}
-tZ_\tau,$$
where $\a\oy=\oy(-Z)$ and $\a_\tau\o_{Y_\tau}=\o_{Y_\tau}(-Z_\tau)$.
Now take an exceptional prime divisor $E$ over $X|_{T^*}$ with the center $\Sigma$, then $E_\tau$ is the disjoint sum of non singular exceptional divisors $E_\tau^{(i)}$ with the center $\sigma(\tau)$ and 
$$\ord_E(\hk_{Y/X}-J_{Y/X}-tZ)=\ord_{E_\tau^{(i)}}(\hk_{{Y_\tau}/X_{\tau}}-J_ {{Y_\tau}/X_{\tau}}-Z_{\tau}).$$
Hence, the constancy of the MJ-minimal log discrepancy follows as required.

For the lower semi continuity of MJ-minimal log discrepancy follows just by showing
\begin{equation}\label{semiconti} \hmld(\sigma(0), X_0, \a^t_0)\leq \hmld(\sigma(\tau), X_\tau, \a^t_\tau)
\end{equation}
for some $\tau\in T^*$.

As in the same way to get (\ref{description1}) in the proof of Theorem \ref{deform-cano},
we obtain
$$\hmld(\sigma(\tau); X_\tau,{\widetilde\a_\tau^t})=\inf_{m,n}\{(M+1)N-(m+1)t-(n+1)c $$
$$\ \ \ \ \ \ \ \ \ \ \ \ \ \ \ \ \ \ \ \ - \dim\left(\psi_{Mm\tau}^{-1}(\lm(Z_\tau))
\cap \psi_{Mn\tau}^{-1}(\ln(X_\tau))\cap \psi_{M0\tau}^{-1}(\sigma(\tau))\right)\},$$
where $M=\max\{m, n\}$ and $ \psi_{mn\tau}:\lm(A_\tau)\to \ln(A_\tau)$ is the canonical
projection.
For simplicity, let us assume $M=n$. (For the other case $M=m$, the proof is the 
same).
Then the scheme $ \psi_{Mm}^{-1}(\lm(Z_\tau))
\cap \psi_{Mn}^{-1}(\ln(X_\tau))\cap \psi_{M0}^{-1}(\sigma(\tau))$
is the fiber of the point $\sigma(\tau)$ by the canonical projection
$$\rho_{nm}: \mathcal{W}_{nm}:=\psi_{nm}^{-1}(\lm(Z))\cap\ln(X/T)\to \Sigma\simeq T,$$
where $ \psi_{nm}:\ln(A)\to \lm(A)$ is the canonical
projection.

Here, note that the space $\mathcal{W}_{nm}$ is $\bG_m$-invariant and
 also the subspace $S_r:=\{Q\in \mathcal{W}_{nm}\mid
 \dim  \rho_{nm}^{-1}\rho_{nm}(Q)\geq r\}$ is 
$\bG_m$-invariant for every $r\in \bN$.
For every $r\in \bN$, the subset $S_r$ is known to be a closed subset 
(cf., for example, \cite[Chapter 1, \S 8]{mu}).
Therefore by 
\cite[Proposition 3.2]{cr}, 
$$\{\tau\in T\mid \dim \rho_{nm}^{-1}(\tau)\geq r\}=\rho_{nm}(S_r)$$
is a closed subset of $T$.
Therefore, for fixed $m, n\in \bN$ \\
$$\tau \mapsto d_{nm}(\tau):=(M+1)N-(m+1)t-(n+1)c
 - \dim \rho_{nm}^{-1}(\tau)$$
is lower semi continuous.
Therefore, there is a non empty open subset $U_{nm}\subset T^*$ such that
$d_{nm}(0)\leq d_{nm}(\tau)$ for all $\tau\in U_{nm}$.
As $k$ is uncountable, $\bigcap_{nm}U_{nm}\neq \emptyset$ which completes the
proof of (\ref{semiconti}).

\end{proof}

\section{Low dimensional MJ-singularities}

In this section we determine MJ-canonical and MJ-log canonical singularities 
of dimension 1 and 2.

\begin{prop}\label{1dim}
Let $(X,x)$ be a singularity on a one-dimensional reduced scheme. Then the following hold:
\begin{enumerate}
\item[(i)] $(X,x)$ is MJ-canonical if and only if it is non singular.
\item[(ii)] $(X,x)$ is MJ-log canonical if and only if it is non singular or ordinary node.
\end{enumerate}
\end{prop}

\begin{proof}
It is clear that a non singular point is MJ-canonical. 
On the contrary if $(X,x)$ is MJ-canonical, then it must be normal by Proposition 
\ref{cano=normal}. 
We can  see the non singularity of $(X,x)$ also by 
$\emb\leq 2\dim X-1=1$ (Proposition \ref{emb})

For (ii), assume $(X,x)$ is singular, then it is MJ-log canonical if and only if 
$\hmld(x;X,\ox)=0$ by \cite[Corollary 3.15]{Ishii} and it is equivalent to 
that $(X,x)$ is ordinary node by \cite{ir}.
\end{proof}
\begin{exmp}\label{duboiscurve}
 It is known that the union of the three axes  in the 3-dimensional affine space
is a Du Bois curve. 
But it is not an  MJ-log canonical curve by Theorem \ref{1dim}, (ii).
\end{exmp}

\begin{thm}\label{rdp}
Let $(X,x)$ be a singularity on 2-dimensional reduced scheme. Then $(X,x)$ is MJ-canonical
if and only if it is non singular or rational double.
\end{thm}

\begin{proof}
First note that for a complete intersection singularity, canonicity and MJ-canonicity are
equivalent.
As a 2-dimensional rational double point $(X,x)$ is a hypersurface singularity and 
canonical, therefore
it is MJ-canonical.
Conversely, if $(X,x)$ is MJ-canonical, then $\hmld(x;X,\ox)\geq 1$.
Such singularities are classified in \cite{ir} to be non singular or rational double or normal crossing double or a pinch point.
As an MJ-canonical singularity is normal by Proposition \ref{cano=normal}, only
rational double points among them can be MJ-canonical.
\end{proof}

Next we will characterize MJ-log canonical singularities of dimension 2.
By Proposition \ref{emb}, for an MJ-log canonical singularity $(X,x)$ of dimension 2,
we have $$\emb(X,x)\leq 4.$$ 
First we will determine the case $\emb(X,x)=3$.
Many of the singularities listed in the following theorem can be observed to be MJ-log canonical singularities by the calculation in \cite{ku}.
But we give a self contained proof below.

\begin{thm}\label{hyp3}
Let $(X,0)$ be a singularity on a 2-dimensional reduced scheme with  $\emb(X, 0)=3$.  
Then, 
$(X, 0)$ is an MJ-log canonical singularity  if and only if   $X$ is defined by 
$f(x,y,z)\in k[[x,y,z]]$ as follows:  
\begin{enumerate}
\item[(i)] $\mult_0f=3$ and the projective tangent cone of $X$ at 0 is a reduced curve with at worst ordinary nodes. 
\item[(ii)] $\mult_0 f=2$  
\begin{enumerate}
\item $f=x^2+y^2+g(z)$, $\deg g\geq 2$.  
\item $f=x^2+g_3(y,z)+g_4(y,z) $, $\deg g_i\geq i$, $g_3$ is homogeneous of degree 3 and $g_3\neq l^3 $ ($l$ linear)  
\item $f=x^2+y^3+yg(z)+h(z)$, $\mult_0g\leq 4$ or $\mult_0h\leq 6$.  
\item $f=x^2+g(y,z)+h(y,z)$, $g$ is homogeneous of degree 4 and it does not have a linear factor with multiplicity more than 2.  
\end{enumerate}
\end{enumerate}
\end{thm}

\begin{proof} Let $(X,0)$ be an MJ-log canonical singularity defined by $f\in k[[x,y,z]]$. 
By (\ref{hmld}) in Proposition \ref{description}, we have
$$\hmld(0;X,\ox)=\inf_{n}\{(n+1)2-\dim(\psi^X_{n0})^{-1}(0)\}\geq 0,$$
therefore in particular for $n=3$, we have
$$\dim(\psi^X_{3,0})^{-1}(0)\leq 8.$$
Here, as $(\psi^X_{3,0})^{-1}(0)=\spec k[x^{(i)},y^{(j)},z^{(k)}\mid i,j,k=1,2,3]/ (F^{(1)},F^{(2)},F^{(3)})$, at least one of $F^{(j)}$ $(j=1,2,3)$ must be non zero
in $k[x^{(i)},y^{(j)},z^{(k)}]$.
By Remark \ref{keisan}, this implies that $\mult_0f\leq 3$.

{\bf Case I}: $\mult_0f=3$ 

Let $(X,0)\subset (A, 0)$ be the embedding into the 3-dimensional non singular
variety, and let $\Phi: A'\to A$ be the blow-up at 0.
Let $E$ be the exceptional divisor on $A'$, $X'$ the strict transform of $X$ in $A'$,
 $\Psi:\ola \to A'$ a factorizing resolution of $X'$ in $A'$ and $\olx$ the strict 
 transform of $X'$ in $\ola$.
We can take $\Psi$ such that the restriction $\psi=\Psi|_{Y}:Y\to X'$ is a log 
resolution of $\j_{X'}\j_X\ox_{'} $.   
As $X$ is a hypersurface of multiplicity 3
at 0, we have 
	$$I_X{\oa}_{'}=I_Y{\oa}_{'}(-3E).$$
Then, by Corollary \ref{blowup}, it follows
$$\hk_{\olx/X'}-J_{\olx/X'}-\psi^*(E|_{X'})=\hk_{\olx/X}-J_{\olx/X}.$$
Therefore, $(X,0)$ is MJ-log canonical if and only if $(X', E|_{X'})$ is MJ-log canonical around $E|_{X'}$.
Since $X'$ is a hypersurface, it is $S_2$, then by Corollary \ref{stin}, (ii), MJ-log canonicity of $(X', E|_{X'})$  is equivalent to that $E|_{X'}$ is reduced and MJ-log canonical.
As $\dim (E|_{X'})=1$ we can apply Proposition \ref{1dim}, (ii), and obtain that $E|_{X'}$
has ordinary nodes. Note that $E|_{X'}$ is a hypersurface in $\bP^2$ defined by
$\int(f)$.

{\bf Case II} $\mult_0f=2$

Let $\Phi:A'\to A$  be the blow-up at 0, $X'$ the strict transform of $X$ in $A'$ and $E$ the exceptional
divisor with respect to $\Phi$.
Then as the same discussion using Corollary \ref{blowup} as in (I), it follows that
$X$ has MJ-log canonical singularities if and only if $X'$ has MJ-log canonical singularities
along $E$.

Here we introduce an invariant for a hypersurface singularity.
The smallest possible dimension $\tau(f)$
of a linear subspace $V_0$ of $V=k x+ ky + k z$ such
that $\mbox{in} (f) $ lies in the subalgebra $k[V_0]$ of $k[x,y,z]$ is an invariant of the germ $(X, {0})$ (\cite[3.15]{ir}).
(In particular for $\mult_0f=2$, $\tau$ is just the rank of the 
quadratic forms defining the tangent cone, therefore it is clear that
$\tau$ is an invariant of $(X,x)$.)

(II-1) $\tau(f)\geq 2$

In this case,  by Weierstrass preparation theorem
and a coordinate transformation (for example, see \cite{ir}) the equation $f=0$ is written as:
$$x^2+y^2+g(z)=0,$$
where $\mult_0 g\geq 2$ ( if $g=0$ we define $\mult_0 g=\infty$).
In this case $\hmld(0; X,\ox)=1$ by \cite{ir}, therefore $(X,0)$ is MJ-log canonical.

(II-2) $\tau(f)=1$

In this case the equation $f=0$ is written as:
$$x^2+g(y,z)=0,$$
where $\mult_0 g\geq 3$.
Now let us consider the germ of the hypersurface $g(y,z)=0$ at 0 in $\spec k[[y,z]]$.
Although this germ
depends on the choice of  the coordinates, 
its multiplicity
$m_2:={\mult} \ g$, and its $\tau$-invariant at 0, let it be $\tau_2$, only depends on $(X,0)$
(this
follows from \cite{hir}. See \cite[Remark 3.19]{ir} ).

(II-2-1) $\tau(f)=1, m_2\geq 5$

In this case $(X,x)$ is not MJ-log canonical.
Indeed we can see that ${\bf{1}}=(1,1,1) \not\in \Gamma(f)$, which implies that $(X, 0)$ is not MJ-log canonical by Corollary \ref{degenerate}.

(II-2-2) $\tau(f)=1, m_2=4$

In this case the equation $f$ is written as
$$x^2+g_4(y,z)+g_5(y,z)=0,$$
where $g_4$ is homogeneous of degree 4 and $\mult_0 g_5\geq 5$.
Then, we can see that the singular locus $C$ of $X'$ lying on $E$ is isomorphic to $\bP^1$.
Let $\Phi':A''\to A'$ is the blow-up with the center $C$, $X''$ the strict transform of $X$ in $A''$ and
$F$ the exceptional divisor with respect to $\Phi'$.
Then, as $I_{X'}\o_{A''}=I_{X''}\o_{A''}(-2F) $ and $K_{A''/A'}=F$, by Theorem \ref{comparison} we obtain
$$\hk_{\olx/{X''}}-J_{\olx/{X''}}-{\Psi'}^*(F|_{X''})=\hk_{\olx/{X'}}-J_{\olx/{X'}},$$
where $\Psi': \ola\to A''$ is a factorizing resolution of $X''$ in $A''$ and $\olx$ is the
strict transform of $X''$ in $\ola$.
The above equality yields the $X'$ has MJ-log canonical singularities if and only if $(X'', F|_{X''})$ is MJ-log
canonical.
Here, as $X''$ is a hypersurface, so in particular satisfies $S_2$ condition, by Corollary \ref{stin} the curve $F|_{X''}$ is reduced and MJ-log canonical.
We can see that $F|_{X''}$ has at worst ordinary nodes if and only if $g_4$ does not have 
 a linear factor with multiplicity more than 2.  
 
 (II-2-3) $\tau(f)=1, m_2=3$
 
 (II-2-3-a) $\tau(f)=1, m_2=3, \tau_2>1$
 
 In this case it is proved that $\hmld(0; X,\ox)=1$ in \cite[Proposition 3.21]{ir}.
 Therefore $(X,0)$ is MJ-log canonical.
 
 (II-2-3-b)  $\tau(f)=1, m_2=3, \tau_2=1$
 
 In this case the equation $f$ is written as
 
 $$f=x^2+y^3+yg(z)+h(z),$$
 where $\mult_0g\geq 3$ and $\mult_0h\geq 4$.
 
 If $\mult_0g=3$ or $\mult_0h\leq 5$, then $\hmld(0; X,\ox)=1$ by \cite[Proposition 3.23]{ir}.
 Therefore $(X,0) $ is MJ-log canonical.
 
 If $\mult_0g=4$ or  $\mult_0h= 6$, by a coordinate transformation we may assume 
 $g(z)=az^4$ and $h(z)=bz^6+\mbox{(higher\ degree\ term\ in\  }z$) ($a,b \in k$).
 Here, note that the condition ``$\mult_0g=4$ or  $\mult_0h= 6$" implies ``$a\neq 0$ or $b\neq 0$".
 Take a blow-up $\Phi:A'\to A$ and look at the equation defining $X'$ on each canonical affine chart of $A'$,
 we can see that on two affine charts $X'$ is non singular and on one affine chart $X'$ is 
 defined by 
 $$u^2+v^3w+avw^3+bw^4+ h'(w) =0,$$
 where $\mult_0h'\geq 5$.
Here, as $a\neq 0$ or $b\neq 0$, the degree 4 part $v^3w+avw^3+bw^4$ does not have 
a linear factor with multiplicity 3.
Therefore, by (II-2-2) the singularity is MJ-log canonical at the point with the coordinate $(u,v,w)=(0,0,0)$ 
 and the other points are non singular.
Thus, in this case $(X,0)$ is MJ-log canonical.

If $\mult_0g\geq 5$ and $\mult_0h\geq 7$, then the Newton polygon $\Gamma(f)$
does not contain the point ${\bf 1}=(1,1,1)$.
Therefore by Corollary \ref{degenerate} the singularity $(X,0)$ is not MJ-log canonical.
\end{proof}

Next we consider the case $\emb(X,0)=4$.

\begin{lem}\label{emb4}  Assume that $X$ is 2-dimensional MJ-log canonical at a point $0\in X$ with 
$\emb(X,0)=4$. 
Then the following hold:
\begin{enumerate}
\item[(i)]
When we write $\widehat{\ox_{,0}}\simeq k[[x_1,x_2,x_3,x_4]]/I$,
the ideal $I$ contains two elements $f,g$ with $\mult_0f=\mult_0g=2$ and
$\int(f), \int(g)$ form a regular sequence in $k[x_1,x_2,x_3,x_4]$.
\item[(ii)] The projective scheme $E_X:=V(\int(I))\subset \bP^3$ is a reduced
curve with at worst ordinary nodes.
\end{enumerate}
\end{lem}

\begin{proof}
By (\ref{hmld}) in Proposition \ref{description}, we have
$$\hmld(0;X,\ox)=\inf_{n}\{(n+1)2-\dim(\psi^X_{n0})^{-1}(0)\}\geq 0,$$
therefore in particular for $n=2$, we have
\begin{equation}\label{regularsequence}
\dim(\psi^X_{2,0})^{-1}(0)\leq 6.
\end{equation}
Here, note that 
$$(\psi^X_{2,0})^{-1}(0)=\spec k[x_1^{(i)},x_2^{(j)},x_3^{(k)}, x_4^{(l)}\mid i,j,k,l=1,2]/ (F^{(1)},F^{(2)}\mid f\in I)$$ under the notation in Remark \ref{keisan}.
Since 4 is the embedding dimension of $(X,0)$, it follows that $\mult_0 f\geq 2$ for
all $f\in I$, therefore $F^{(1)}=0$ for all $f$ 
by Remark \ref{keisan}.
By the inequality (\ref{regularsequence}) we obtain that there exist $f, g\in I$ such
that $F^{(2)}(x_i^{(1)}), G^{(2)}(x_i^{(1)})$ form a
regular sequence in $k[x_1^{(i)},x_2^{(j)},x_3^{(k)}, x_4^{(l)}\mid i,j,k,l=1,2]$, therefore 
these form a regular sequence in $k[x_1^{(1)},x_2^{(1)},x_3^{(1)},x_4^{(1)}]$.
As $\int(f)(x_i^{(1)})=F^{(2)}, \int(g)(x_i^{(1)})=G^{(2)}$,  
it follows that $\mult_0 f=\mult_0 g=2$ by Remark \ref{keisan} and that  $\int(f), \int(g)$ form a regular sequence in $k[x_1,x_2,x_3,x_4]$. 
This completes the proof of (i).

Now let $A$ be a non singular variety of dimension 4 containing a neighborhood
of the singularity $(X,0)$ and let $A'\to A$ be the blow-up at 0 with the exceptional
divisor $E\simeq \bP^3$. Let $X'\subset A'$ be the strict 
transform of $X$ in $A'$.
Then, note that $E|_{X'}=E_X$ and we have  
$$I_X{\oa}_{'}\subset I_{X'}{\oa}_{'}(-2E).$$
By taking a factorizing resolution $\Psi:\ola\to A'$ of $X'$ in $A'$ with the strict 
transform $\olx$ of $X'$, we obtain
\begin{equation}\label{comp}
\hk_{\olx/{X'}}-J_{\olx/{X'}}-\Psi^*E|_\olx \geq \hk_{\olx/{X}}-J_{\olx/{X}}
\end{equation}
by Corollary \ref{blowup}.
Now, by the assumption that $X$ is MJ-log canonical at 0, it follows that 
$(X', E_X)$ is MJ-log canonical, which implies 
$\hmld(y;X', E_X)\geq 0$
for every $y \in E_X$.
Therefore we obtain 
$$\hmld(y;X',\ox_{'})\geq 1.$$
But such a 2-dimensional singularity $(X',y)$ is determined as either non singular or a
hypersurface singularity (see, for example \cite[Lemma 3.6]{ir}).
Hence $X'$ satisfies $S_2$ condition around $E_X$.
Then, by Corollary \ref{stin}, $E_X$ is reduced and MJ-log canonical, which yields
the statement (ii).
\end{proof}

\begin{thm}\label{ci}  Let $(X,0)$ be a singularity on a 2-dimensional reduced scheme with  $\emb(X, 0)=4$.  
Then, the following hold:
\begin{enumerate}
\item[(i)] In case 
$(X, 0)$ is locally a complete intersection:

\noindent
$X$ is MJ-log canonical at 0  if and only if \\
$\widehat{\ox_{,0}}\simeq k[[x_1,x_2,x_3,x_4]]/(f,g)$, where $f,g$ satisfy the 
conditions that  
 $\mult_0f=\mult_0 g=2$  and $V(\int(f),\int(g))\subset \bP^3$ 
is a reduced curve with at worst ordinary nodes.
\item[(ii)] In case
$(X,0)$ is not locally a complete intersection:

\noindent
$X$ is  MJ-log canonical at 0  if and only if $X$ is a closed subscheme of
 a locally  complete intersection scheme $M$ which is MJ-log canonical at 0.

\end{enumerate}

\end{thm}
\begin{proof} For the proof of (i), assume that $(X,0)$ is locally a complete intersection
and $\widehat{\ox_{,0}}\simeq k[[x_1,x_2,x_3,x_4]]/(f,g)$.
Assume that $(X,0)$ is MJ-log canonical.
Then, by Lemma \ref{emb4}  it follows $\mult_0 f=\mult_0 g=2$.
Because in Lemma \ref{emb4} it is proved that $E_X=V(\int(I))$ is a reduced curve with
at worst ordinary nodes, it is sufficient to prove that $V(\int(f),\int(g))=V(\int(I))$.
In general for a complete intersection singularity defined by $f,g$ 
the inequality  $$\mult(X,0)\geq (\mult_0f)( \mult_0g) $$ holds.
Here, note that  $\mult(X,0)=\deg
(V(\int(I))\subset \bP^3)$.
Noting that $V(\int(I))\subset V(\int(f), \int(g))$, we have 
$\deg V(\int(I))\leq \deg V(\int(f),\int(g))$, which implies 
$$\mult(X,0)\leq(\mult_0f_1)(\mult_0f_2).$$ 
Therefore the equalities  hold, in particular  $V(\int(I))= V(\int(f), \int(g))$.

Conversely, if $\widehat{\ox_{,0}}\simeq k[[x_1,x_2,x_3,x_4]]/(f,g)$ and
$f,g$ satisfy the conditions in (i).
The conditions claim that $E_X$ is a MJ-log canonical curve.
By Corollary \ref{stin}, we have $(X', E_X)$ is MJ-log canonical  around $E_X$.
On the other hand, in this case we have
$$I_X{\oa}_{'}= I_{X'}{\oa}_{'}(-2E).$$
Therefore by Corollary \ref{blowup}, we obtain the equality in (\ref{comp})
$$\hk_{\olx/{X'}}-J_{\olx/{X'}}-\Psi^*E|_\olx =\hk_{\olx/{X}}-J_{\olx/{X}},$$
which yields that $X$ is MJ-log canonical at 0.

\vskip.3cm
 
For the proof of (ii), first assume that  $X$ is a subscheme
of an MJ-log canonical 2-dimensional locally complete intersection scheme $M$.
By Adjunction formula in \cite[Corollary 3.12]{Ishii} we have
$$\hmld(0; X,\ox)\geq \hmld(0;M,{\mathcal O}_M).$$
As the right hand side is non negative by the assumption, 
 we obtain that $X$ is MJ-log canonical at 0.

Conversely assume that $X$ is MJ-log canonical at 0. 
Assume also that $X$ is not locally a complete intersection at 0.
Then, by Lemma \ref{emb4}, there are two elements $f,g\in I$ such that 
$\mult_0f=\mult_0g=2$ and $\int(f), \int(g)$ define a curve in $\bP^3$.
Here $I$ is the ideal as in the proof of Lemma \ref{emb4}.
Let $E'=V(\int(f), \int(g))\subset \bP^3$.
Let $\ola\stackrel{\Psi}\longrightarrow A' \to A$, $\olx\to X'\to X$, $E\subset A'$ and $E_X\subset X'$  as in the proof of Lemma \ref{emb4}. 
Then, as $\int(f), \int(g)\in \int(I)$, we have $E_X\subset E'$.
Therefore $\deg E_X\leq \deg E'=4$ in $\bP^3$.
By the assumption that $X$ is not locally a complete intersection at 0, it follows that 
$E_X$ is not 
a complete intersection, therefore 
\begin{equation}\label{deg}
\deg E_X\leq 3.
\end{equation}
On the other hand $E_X$ is reduced and has at worst ordinary nodes
by Lemma \ref{emb4}.
By the result of (i), for the proof of the statement,  it is sufficient to prove that there are two elements $f',g'\in I$ 
such that $V(\int(f'), \int(g'))$ is a reduced curve with at worst ordinary nodes.
Therefore it is sufficient to prove that there exists in $\bP^3$  a complete intersection 
reduced curve $E''$  which contains $E_X$ such that $E''$ has at worst ordinary nodes.
Here, we note that $E_X$ is not a complete intersection, because if it is a complete
intersection, then $X$ is also a complete intersection.

An irreducible curve in $\bP^3$ of degree $\leq 3$ is classified as follows:
\begin{enumerate}
\item[(a)] $\deg C=1 \Leftrightarrow C$ is a line.
\item[(b)] $\deg C=2 \Leftrightarrow C$ is a conic in $\bP^2$.
\item[(c)] $\deg C=3 \Leftrightarrow C$ is either a plane cubic with genus 1 or a twisted
cubic.
\end{enumerate}

Case 1:  The case $\deg E_X=1$ does not happen. 
Because,
if $\deg E_X=1$, then $E_X$ must be irreducible and by (a) it is a line,
therefore $E_X$ is a complete intersection, a contradiction.

Case 2: The case $\deg E_X=2$. 
In this case, the possibility of $E_X$ is as follows:

(1) a plane conic, (2)  the union of two lines which intersect at one point,
(3) the disjoint union of two lines.

\noindent The cases (1), (2) do not happen as $E_X$, because in these cases the curve becomes
a complete intersection.
In case (3), $E_X$ is the union of skew lines, therefore by a suitable coordinate system in $\bP^3$,
we can write $E_X=V(x_1, x_2)\cup V(x_3, x_4)$. Then $E_X$ is contained in a 
complete intersection scheme $V(x_1x_3, x_2x_4)$.
We can see that this scheme is a cycle of four $\bP^1$'s with ordinary  nodes.
We can take this scheme $V(x_1x_3, x_2x_4)$ as $E''$.

Case 3: The case $\deg E_X=3$.
In this case, the possibility of $E_X$ is as follows:

(4)  a plane cubic of genus 1, (5)  a twisted cubic, 
(6) the union of a plane conic and a line,
(7) the union of three lines.

The case (4) does not happen as $E_X$, because in this case the curve is a complete intersection.
If $E_X$ is as in (5), then $E_X$ is defined by $x_1x_3-x_2^2=x_2x_4-x_3^2=x_1x_4-x_2x_3=0$.
Then the complete intersection curve $V(x_1x_3-x_2^2+x_2x_4-x_3^2, x_1x_4-x_2x_3)$ contains $E_X$ and it is reduced and has only ordinary nodes.
So take this scheme as $E''$.

In case (6), 
first we show that the conic $Q$ and the line $l$ intersect.
Let $S$ be a surface defined by a general element in the vector space $\{a(\int(f))+b(\int(g))\mid a,b \in k\}$.
Then $S$ must be an irreducible surface, because otherwise $S$ must be the union of 
two hyperplanes and $E'$ becomes a line, a contradiction.
Therefore $S$ is a cone over a plane conic or non singular.
If $S$ is a cone, then a plane conic on $S$ and a line on $S$ intersect.
If $S$ is non singular, then $S\simeq \bP^1\times\bP^1$ and the lines on $S$ are 
 either of the type $C_p=\{p\}\times \bP^1$ or of the type $D_q=\bP^1\times\{q\}$,
 where $p, q$ are points in $\bP^1$.
 A conic on $S$ is linearly equivalent to $C_p+D_q$ which has a positive intersection 
 number with $C_p$ and $D_q$.
 Now we obtained $Q\cap l\neq \emptyset$.

Here, if the conic and the line lie on a plane, then the curve becomes a complete
intersection. Therefore $E_X$ is not of this type.
Assume that the conic $Q$ and the line $l$ do not lie on a plane.
We can take $Q$ on a hyperplane $x_1=0$. 
By a suitable choice of the coordinate system, we may assume that $l=V(x_2, x_3)$.
Let $g=g(x_2,x_3,x_4)$ be the defining equation of $Q$ in the hyperplane and
$\ell=ax_2+bx_3$ a general linear combination of $x_2$ and $x_3$.
Then the complete intersection scheme $V(g, x_1\ell)$ contains $Q\cup l$ and it is a reduced curve consisting of
a plane conic and two lines $l, l'$ intersecting normally at the point $(1,0,0,0)$ with
ordinary double 
intersection also at $Q\cap l'$.
Therefore if $E_X=Q\cup l$, we can take $V(g, x_1\ell)$ as $E''$.

In case (7), take $S$ as above. 
If $S$ is a cone over a plane conic and if $E_X$ consists of three lines, then
 by $E_X\subset S$ three lines must  intersect at the vertex, therefore it is not 
 ordinary double, which shows that $E_X$ is not of this type.
 If $S$ is non singular, then, as was stated above, a line on $S$ is either of the form
  $C_p$ or $D_q$.
Because of the symmetry of $C$ and $D$, we may assume that the union of three lines on $S$ is either the union of three $C_p$'s or the union of 
 two $C_p$'s and one $D_q$.
 The union of three $C_p$'s is not possible for $E_X$. 
 Because otherwise, $E_X\subset E'$ and $E'=S\cap H$, where $H$ is a hypersurface 
 of degree $2$.
 Then $$3=(E_X\cdot D_q)_S\leq (E'\cdot D_q)_S= H\cdot D_q=2,$$  which is a contradiction.
 Here, $(\ \ \cdot\ \ )_S$ is the intersection number of the divisors on $S$ and 
 $H\cdot D_q$ is the intersection number of the divisor $H$ and a curve $D_q$
 in $\bP^3$.
 
 Now if $E_X$ is the union of  $C_{p_1}, C_{p_2}$ and  $D_q$, then it is a chain of lines and
 by a suitable choice of the coordinate system, these are 
 represented as  $C_{p_1}=V(x_1, x_2), C_{p_2}=V(x_3, x_4)$ and  $D_q=V(x_2,x_3)$.
 Then the complete intersection $V(x_1x_3, x_2x_4)$ contains $E_X$ and $V(x_1x_3, x_2x_4)$ is reduced and has at worst ordinary nodes.
 Thus every possible $E_X$ is contained in a complete intersection curve which is
 reduced and has at worst ordinary nodes.
\end{proof}

\begin{exmp}\label{nonS2} Let $X\subset \bA^4$ be defined by $f=x_1x_3, g=x_2x_4\in k[x_1,x_2,x_3,x_4]$.
Then $\int(f)=f$, $\int(g)=g$  and $V(f,g)$ is a cycle  consisting of four $\bP^1$'s 
such that the intersection of each two components is ordinary double.
Then, by Theorem \ref{ci}, $X$ is MJ-log canonical at 0.
Let $C_i$ $(i=1,2,\ldots, 4)$ be the irreducible component of $V(f,g)$ such that
$C_i\cdot C_{i+1}=1$ for $i=1,\ldots, 4$ and let $C_5:=C_1$.
Note that $X$ is the cone over the reduced projective scheme $\bigcup _{i=1}^4C_i
\subset \bP^3$.

Now take the cone $X_1$ over the reduced projective scheme $C_1\cup C_2\cup C_3
\subset \bP^3$. By Theorem \ref{ci}, $X_1$ is MJ-log canonical at 0.
This  example was proved to be  non semi log canonical singularity  by Koll\'ar \cite[Example 5.16]{ko}.

Next take the cone $X_2$ over the reduced projective scheme $C_1\cup C_3\subset
\bP^3$. By Theorem \ref{ci}, $(X_2,0)$ is also MJ-log canonial.
This is an example of MJ-log canonical singularity but not $S_2$.
Indeed $X_2$ is the union of two irreducible surfaces which intersect at a point 0,
therefore 
$X_2$ does not satisfy $S_2$.

\end{exmp}

\makeatletter \renewcommand{\@biblabel}[1]{\hfill#1.}\makeatother

\end{document}